\documentclass[a4paper,12pt,reqno]{amsart}

\usepackage{fullpage}
\usepackage[all]{xy}

\newcommand{\Tors}{\operatorname{Tors}}

\newcommand{\Bl}{\operatorname{Bl}}
\newcommand{\grmod}{\operatorname{grmod}}

\newcommand{\ann}{\operatorname{ann}}

\newcommand{\rep}{\operatorname{rep}}
\newcommand{\Auteq}{\operatorname{Auteq}}

\usepackage{mymacros}

\usepackage{stmaryrd}

\newcommand{\Wred}{W_{\mathrm{red}}}

\newcommand{\Alg}{\operatorname{Alg}}

\newcommand{\shbimod}[2]{\operatorname{shbimod}(#1, #2)}

\newcommand{\twedge}{\textstyle{\bigwedge}}
\newcommand{\lstar}[1]{\!^* #1}
\newcommand{\lbot}[1]{\!^\bot #1}
\newcommand{\X}{{\bP ^{ 1 }}}

\newcommand{\xc}[1]{\textcolor{red}{#1}}

\newcommand{\cMtilde}{\widetilde{\cM}}
\newcommand{\Msht}{\cMtilde_{\mathrm{sh}}}
\newcommand{\Msh}{\cM_{\mathrm{sh}}}
\newcommand{\Mell}{\cM_{\mathrm{ell}}}

\newcommand{\Mshz}{\cM_{\mathrm{sh},0}}
\newcommand{\Mshtz}{\cMtilde_{\mathrm{sh},0}}
\newcommand{\Mrelz}{\cM_{\mathrm{rel},0}}
\newcommand{\Mreltz}{\cMtilde_{\mathrm{rel},0}}
\newcommand{\Mellz}{\cM_{\mathrm{ell},0}}
\newcommand{\Msho}{\cM_{\mathrm{sh},1}}
\newcommand{\Mshto}{\cMtilde_{\mathrm{sh},1}}
\newcommand{\Mrelto}{\cMtilde_{\mathrm{rel},1}}
\newcommand{\Mrelo}{\cM_{\mathrm{rel},1}}
\newcommand{\Mello}{\cM_{\mathrm{ell},1}}
\newcommand{\Mshi}{\cM_{\mathrm{sh},i}}
\newcommand{\Mreli}{\cM_{\mathrm{rel},i}}
\newcommand{\Melli}{\cM_{\mathrm{ell},i}}
\newcommand{\Mellt}{\cMtilde_{\mathrm{ell}}}

\title{Moduli of noncommutative Hirzebruch surfaces}

\author{Izuru Mori}
\address{
Department of Mathematics, 
Shizuoka University,
836 Ohya, 
Suruga-ku, 
Shizuoka 422-8529,
Japan
}
\email{mori.izuru@shizuoka.ac.jp}

\author{Shinnosuke Okawa}
\address{
Department of Mathematics,
Graduate School of Science,
Osaka University,
Machikaneyama 1-1,
Toyonaka,
Osaka,
560-0043,
Japan.
}
\email{okawa@math.sci.osaka-u.ac.jp}

\author{Kazushi Ueda}
\address{
Graduate School of Mathematical Sciences,
The University of Tokyo,
3-8-1 Komaba,
Meguro-ku,
Tokyo,
153-8914,
Japan.}
\email{kazushi@ms.u-tokyo.ac.jp}

\begin{document}

\maketitle

\begin{abstract}
We introduce three non-compact moduli stacks
parametrizing noncommutative deformations of Hirzebruch surfaces;
the first is the moduli stack of locally free sheaf bimodules
of rank 2,
which appears in the definition of noncommutative $\mathbb{P}^1$-bundle
in the sense of Van den Bergh \cite{MR2958936},
the second is the moduli stack of relations of a quiver
in the sense of \cite{1411.7770},
and
the third is the moduli stack of quadruples
consisting of an elliptic curve and three line bundles on it.
The main result of this paper shows that they are naturally birational to each other.
We also give an Orlov-type semiorthogonal decomposition
for noncommutative $\mathbb{P}^1$-bundles,
an explicit classification
of locally free sheaf bimodules of rank 2, and
a noncommutative generalization
of the (special) McKay correspondence as a derived equivalence
for the cyclic group $\left\langle \frac{1}{d}(1,1) \right\rangle$.
\end{abstract}

\tableofcontents

\section{Introduction}
A \emph{noncommutative $\X$-bundle}
over a smooth
scheme $X$
is a Grothendieck category
of the form $\Qgr \bS(\cE)$,
where $\bS(\cE)$ is the noncommutative symmetric algebra
over a locally free sheaf bimodule $\cE$ of rank 2
on $X$
\cite{MR2958936} (see \pref{sc:recap} for a recap).
%
%
%
One motivation to study such objects
comes from Artin's conjecture
\cite[Conjecture 4.1]{MR1477464},
which states that any noncommutative surface is birational to
either
\begin{enumerate}
 \item
a noncommutative projective plane,
 \item
a noncommutative $\X$-bundle over a commutative curve, or
 \item
a noncommutative surface which is finite over its center.
\end{enumerate}

If $X$ is a smooth projective curve of genus $g \ge 2$ and
$\cE$ is a stable vector bundle of rank 2 on $X$,
then one can easily show
$
 H^0 \lb \bP ( \cE ), \twedge^2 \cT_{\bP(\cE)} \rb = 0
$
and
$
 H^2( \bP ( \cE ), \cO_{\bP(\cE)}) = 0,
$
so that
$
 \HH^2(\bP(\cE)) \simeq H^1( \bP ( \cE ), \cT_{\bP(\cE)})
$
which means that every 
formal
noncommutative
deformation of
$
 \bP _{ X } ( \cE )
$
is commutative.

In contrast, any `continuous' noncommutative deformation of a Hirzebruch surface
$
 \Sigma_d \coloneqq \bP _{ \X } ( \cO_\X \oplus \cO_\X(d) )
$
is strictly noncommutative,
since isomorphism classes of
commutative Hirzebruch surfaces are parametrized by natural numbers,
and hence discrete.
It is shown in \cite[Theorem 1.3]{MR2958936}
that any noncommutative deformation, i.e. flat deformation of the
abelian category of quasi-coherent sheaves in the sense of \cite{MR2238922},
of a commutative Hirzebruch surface over a complete Noetherian local ring
is a noncommutative $\X$-bundle over $\X$.

For
$
 d \in \bZ _{ \ge 0 }
$,
the dimension of the Hochschild cohomology of
$
 \Sigma _{ d }
$
is given by
\begin{align}
 \dim \HH^1(\Sigma_d)
  &= h^0(\Theta_{\Sigma_d})
  = \max \{ d-1, 0 \} + 6, \\
 \dim \HH^2(\Sigma_d)
  &= h^0 \lb \twedge^2 \Theta_{\Sigma_d} \rb + h^1(\Theta_{\Sigma_d})
  = (\max \{ d-3, 0 \} + 9) + \max \{ d-1, 0 \}, \\
 \dim \HH^3(\Sigma_d)
  &= h^2(\Theta_{\Sigma_d}) + h^1\lb \twedge^2 \Theta_{\Sigma_d} \rb
  = 0 + \max \{ d-3, 0 \},
\end{align}
so that
the expected dimension
of the moduli space of noncommutative Hirzebruch surfaces
is
\begin{align}
 - \dim \HH^1(\Sigma_d)
 + \dim \HH^2(\Sigma_d)
 - \dim \HH^3(\Sigma_d)
 = 3.
\end{align}

In contrast with the case of del Pezzo surfaces, noncommutative deformation of $\Sigma_d$ is obstructed
when
$
 d > 3
$
by
\cite[Remark 5 and Proposition 9]{goto2016unobstructed}.
The non-vanishing of
$
 h^1 \lb \twedge^2 \Theta_{\Sigma_d} \rb
  = h^1 \lb \cO_{\Sigma_d} \lb - K_{\Sigma_d} \rb \rb
$
also leads to the absence of a geometric helix (in the sense of \cite{MR2646308}) of vector bundles.

The first result of this paper,
given in \pref{sc:sbc},
is a structure theorem
and an explicit and complete classification
of locally free sheaf bimodules of rank
$
 2
$
on
$
 \X
$.
We also study their Gieseker stability
with respect to the anti-canonical bundle in \pref{sc:stability_and_deformation_theory_of_sheaf_bimodules}.
It turns out that many of them are semi-stable,
which implies the unobstructedness of their deformations
as explained in \pref{pr:ext_groups_of_sheaf_bimodules}.

The second result of this paper,
given in \pref{sc:SOD},
is an Orlov-type semi-orthogonal decomposition
of the bounded derived category
$
 D ^{ b } \qgr \bS ( \cE )
$;
the derived category decomposes into two copies of
$
 D ^{ b } \coh X
$,
and the sheaf bimodule
$
 \cE
$
gives the integral kernel of the dual gluing functor.
%
In \pref{sc:fsec}, we show that a certain exceptional collection \eqref{eq:EC} on
$
 D ^{ b } \qgr \bS ( \cE )
$
is strong, for an appropriate choice of a parameter
$
 m
$.
This follows from the computation in \pref{sc:sbc} of the action of the dual gluing functor
$
 \phi ^{ ! }
$ on the line bundles
$
 \cO _{ \X }
$
and
$
 \cO _{ \X } ( - 1 )
$.
In \pref{cr:strong_for_m=1}, we give the list of sheaf bimodules of degree either
$
 2
$
or
$
 1
$
for which the exceptional collection \eqref{eq:EC} with $m = 1$ is strong.
The total morphism algebra of the collection is isomorphic
to the quotient of the path algebra of the quiver
$
 Q ^{ 0 }
$
or
$
 Q ^{ 1 }
$,
shown in \pref{fg:quiver0} or
\pref{fg:quiver1} respectively,
by a two-sided ideal $I$ called \emph{relations of the quiver}.

The relations $I$ depend on the isomorphism class of $\cE$,
and it is natural to ask to which extent $\cE$ can be reconstructed from $I$.
%
%
%
%
%
%
To study this kind of problems,
we introduce three moduli stacks
parametrizing noncommutative Hirzebruch surfaces in suitable senses.

The first moduli stack,
introduced in \pref{sc:Msh},
is the quotient
$\Msh \coloneqq [\Msht/(\PGL_2)^2]$
of the stack $\Msht$ of locally free sheaf bimodules of rank 2 on $\X$
by the natural action of the group
$
 (\PGL_2)^2 \simeq \Aut(\X) \times \Aut(\X)
$.
It is natural to take the quotient by this action, because it preserves the equivalence classes of the associated category
$
 \Qgr \bS ( \cE )
$
as we
recall
in \pref{sc:twisting}.
The stack $\Msh$ is divided into connected components by the degree,
but only the parity of the degree matters,
since components whose degrees are different by a multiple of 2 are isomorphic.
We write the connected component
parametrizing sheaf bimodules of degree 2 (resp. 1) as
$\Mshz$ (resp. $\Msho$).


The second moduli stack,
introduced in \pref{sc:quiver},
is the moduli stack of relations of the quivers
$Q^0$ and $Q^1$ in the sense of \cite{1411.7770}. They will be denoted by $\Mrelz$ and $\Mrelo$, respectively.


The third moduli stack,
introduced in \pref{sc:Mell},
is the moduli stack $\Mell$ of \emph{nonsingular admissible quadruples},
which parametrizes isomorphism classes of collections $(E, L_0, L_1, L_2)$
of an elliptic curve $E$
and three line bundles $(L_0, L_1, L_2)$
such that
$
 L_0 \not \simeq L_2
$.
We consider two connected components $\Mellz$ and $\Mello$ such that
$
 \deg (L_0, L_1, L_2) = ( 2, 2, 2 )
$
and
$
 ( 2, 1, 2 )
$
respectively.

\begin{figure}
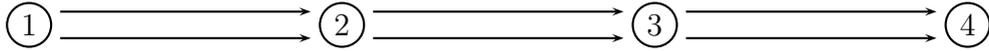

\begin{align*}
\begin{psmatrix}[rowsep=35mm,colsep=35mm,mnode=circle]
 1 & 2 & 3 & 4
\psset{nodesep=3pt,arrows=->}
\ncline[offset=5pt]{1,1}{1,2}
\ncline[offset=-5pt]{1,1}{1,2}
\ncline[offset=5pt]{1,2}{1,3}
\ncline[offset=-5pt]{1,2}{1,3}
\ncline[offset=5pt]{1,3}{1,4}
\ncline[offset=-5pt]{1,3}{1,4}
\end{psmatrix}
\end{align*}
\caption{The quiver $Q^0$}
\label{fg:quiver0}
\end{figure}

\begin{figure}
\ \\[15mm]
\begin{align*}
\begin{psmatrix}[rowsep=35mm,colsep=35mm,mnode=circle]
 1 & 2 & 3 & 4
\psset{nodesep=3pt,arrows=->}
\ncline[offset=5pt]{1,1}{1,2}
\ncline[offset=-5pt]{1,1}{1,2}
\ncline[offset=0pt]{1,2}{1,3}
\ncline[offset=5pt]{1,3}{1,4}
\ncline[offset=-5pt]{1,3}{1,4}
\nccurve[angleA=50,angleB=130,offset=3pt]{1,1}{1,3}
\nccurve[angleA=-50,angleB=-130,offset=3pt]{1,2}{1,4}
\end{psmatrix}
\end{align*}
\ \\[15mm]
\caption{The quiver $Q^1$}
\label{fg:quiver1}
\end{figure}

The third result of this paper
is the comparison of these three moduli stacks:

\begin{theorem} \label{th:main}
For each $i \in \{ 0,1 \}$,
three moduli stacks $\Mshi$, $\Mreli$ and $\Melli$ are
naturally birational to each other. When
$
 i = 0
$,
for a generic triple
\begin{align}
 \lb \cE, I \subset \bfk Q ^{ 0 }, \lb E, L _{ 0 }, L _{ 1 }, L _{ 2 } \rb \rb
\end{align}
which correspond to each other by the birational maps,
there are derived equivalences
\begin{align}
 D ^{ b } \qgr \bS \lb \cE \rb
 \simeq
 D ^{ b } \module \bfk Q ^{ 0 } / I
 \simeq
 D ^{ b } \qgr A \lb E, L _{ 0 }, L _{ 1 }, L _{ 2 } \rb,
\end{align}
where
$
 A \lb E, L _{ 0 }, L _{ 1 }, L _{ 2 } \rb
$
is the 3-dimensional cubic AS-regular $\bZ$-algebra associated to the quadruple
$
 \lb E, L _{ 0 }, L _{ 1 }, L _{ 2 } \rb
$
by the correspondence of \cite{MR2836401}.
\end{theorem}

It is an interesting problem to find a complete moduli scheme
(or a proper Deligne--Mumford stack)
birational to the above moduli stacks.

Recall that
$
 \Sigma _{ d }
$
for $d \ge 2$
admits a birational morphism to the weighted projective plane
$
 \bfP ( 1, 1, d )
$
which contracts the negative section to a
$
 \frac{1}{d} ( 1, 1 )
$-singularity.
As is well-known,
there is a fully faithful functor
$
 D ^{ b } \coh \Sigma _{ d } \hookrightarrow D ^{ b } \coh \bP ( 1, 1, d )
$,
which is a derived equivalence
when $d=2$.
This phenomenon is a particular case of
the (special) McKay correspondence
as a derived equivalence.
In \cite{MR1743651},
Stephenson studied AS-regular algebras $S$ generated by 3 elements of degree
$
 1, 1, d
$.
By taking the associated category
$
 \Qgr S
$
of those algebras, we obtain noncommutative
$
 \bP ( 1, 1, d )
$.
In \pref{sc:McKay},
we consider one family of such algebras $S$,
and show the existence of sheaf bimodules
$
 \cE
$
and fully faithful functors
$
 D ^{ b } \qgr \bS ( \cE ) \hookrightarrow D ^{ b } \qgr S
$,
which are equivalences if $d=2$.



\section*{Notation and conventions}
Throughout this paper,
we fix an algebraically closed base field $\bfk$ of characteristic zero.
Schemes are defined over $\bfk$,
and categories and functors are linear over $\bfk$.
Triangulated categories and exact functors are
enhanced in the sense of \cite{MR1055981}.
For a pair $(X, Y)$ of objects
in an enhanced triangulated category $\cD$,
the complex of $\bfk$-vector spaces
underlying $\Hom_\cD(X, Y)$ will be denoted by $\hom_\cD(X, Y)$.

\section*{Acknowledgements}
During the preparation of this paper,
I.~M.~was partially supported by JSPS Grant-in-Aid for Scientific Research
(16K05097,16K13743,16H03923),
S.~O.~was partially supported by JSPS Grant-in-Aid for Young Scientists No.~25800017,
Grants-in-Aid for Scientific Research
(16H05994,
16K13746,
16H02141,
16K13743,
16K13755,
16H06337)
and the Inamori Foundation, and
K.~U.~was partially supported by JSPS Grant-in-Aid for Scientific Research
(15KT0105,
16H03930,
16K13743).

\section{Recapitulation on noncommutative $\X$-bundle}
\label{sc:recap}
\subsection{Sheaf bimodules}
 \label{sc:shbimod}

Let $X$ and $Y$ be schemes.
A coherent sheaf $\cE$ on $X \times Y$ is said to be
a \emph{sheaf bimodule}
if the scheme-theoretic support
\begin{align}
 W := \Spec _{ X \times Y} \Image \lb \cO _{X \times Y} \to \cEnd \cE \rb
\end{align}
is finite over both $X$ and $Y$
\cite{MR1409223, MR2716732, MR1793588}.
The full subcategory of $\coh(X \times Y)$
consisting of sheaf bimodules is denoted by
$\shbimod{X}{Y}$.
A sheaf bimodule $\cE$ defines a right exact functor
\begin{align}\label{eq:functor_defined_by_the_sheaf_bimodule}
 (-) \otimes_{\cO_X} \cE \colon \coh X \to \coh Y, \quad
 \cF \mapsto p_{Y*} \lb p_X^* \cF \otimes_{\cO_{X \times Y}} \cE \rb,
\end{align}
where
$
 p_X \colon X \times Y \to X
$
and
$
 p_Y \colon X \times Y \to Y
$
are the projections.
Let $u$ and $v$ be the restrictions of $p_X$ and $p_Y$ to $W$ respectively,
so that the inclusion is given by $\iota = u \times v \colon W \to X \times Y$;
\begin{align}
\begin{psmatrix}[rowsep=5mm,colsep=20mm]
  & W \\[6mm]
  & X \times Y \\
  X & & Y.
\end{psmatrix}
\psset{shortput=nab, nodesep=1mm, arrows=->}
\ncline[arrows=H->,hooklength=2mm,hookwidth=2mm,offset=1mm,nodesep=2mm]{1,2}{2,2}_[npos=.7]{\iota}
\ncline{1,2}{3,1}_u
\ncline{1,2}{3,3}^v
\ncline{2,2}{3,1}^{p_X}
\ncline{2,2}{3,3}_{p_Y}
\end{align}
Let $\cU$ be the unique coherent $\cO_W$-module
such that $\iota_* \cU = \cE$.
Then the functor \eqref{eq:functor_defined_by_the_sheaf_bimodule} can be rewritten as
\begin{align} \label{eq:otimesE}
 (-) \otimes_{\cO_X} \cE
  &\simeq v_* \lb u^*(-) \otimes_{\cO_W} \cU \rb
  \colon \coh X \to \coh Y.
\end{align}

The \emph{convolution} of sheaf bimodules
$\cE \in \shbimod{X}{Y}$ and $\cF \in \shbimod{Y}{Z}$
is defined by
\begin{align}
 \cE \otimes_{\cO_Y} \cF
  \coloneqq p_{XZ*}
   \lb p_{XY}^* \cE \otimes_{\cO_{X \times Y \times Z}}
    p_{YZ}^* \cF
   \rb
    \in \shbimod{X}{Z},
\end{align}
where
$
 p _{ \bullet \circ }
$
denotes the projection from
$
 X \times Y \times Z
$
to
$
 \bullet \times \circ
$.
It is clear that there exists an isomorphism
\begin{align}
 \lb (-) \otimes_{\cO_X} \cE) \otimes_{\cO_Y} \cF \rb
  \simeq (-) \otimes_{\cO_X} \lb \cE \otimes_{\cO_Y} \cF \rb
\end{align}
of functors from $\coh X$ to $\coh Z$.

A sheaf bimodule $\cE$ is said to be \emph{locally free of rank $r$}
if both $p_{X*} \cE$ and $p_{Y*} \cE$ are locally free of rank $r$
\cite[Definition 3.1.3]{MR2958936}.
The structure sheaf $\cO_{\Delta_X}$
of the diagonal $\Delta_X \subset X \times X$
is called the \emph{diagonal sheaf bimodule}.
It is locally free of rank one, and
the corresponding functor
$
 (-) \otimes_{\cO_X} \cO_{\Delta_X} \colon \coh X \to \coh X
$
is isomorphic to the identity functor.

The locally-freeness of a sheaf bimodule implies that the integral transformation is defined on the underived level.

\begin{lemma}
Suppose that a sheaf bimodule $\cE$ is locally free. Then the functor \eqref{eq:otimesE} is exact.
\end{lemma}


The functor \eqref{eq:otimesE}
extends to an exact functor of triangulated categories from $D^b \coh X$ to $D^b \coh Y$,
which we write $(-) \otimes_{\cO_X} \cE$ again by an abuse of notation.

It is known \cite[Section 3]{MR2958936} that if
$
 \cE
$
is locally free, then the functor \eqref{eq:functor_defined_by_the_sheaf_bimodule} admits the left and the right adjoint functors. They are given by
$
 (-) \otimes_{\cO_Y} \lstar{\cE} \colon \coh Y \to \coh X
$
and
$
 (-) \otimes_{\cO_Y} \cE^* \colon \coh Y \to \coh X
$
where
\begin{align}
 \cE^* &\coloneqq
 p _{ Y }^* \omega_{Y}^{-1} \otimes_{\cO_{Y \times X}} \cE^D, \\
 \lstar{\cE}
  &\coloneqq \cE^D \otimes_{\cO_{Y \times X}} p _{ X }^* \omega_X^{-1},
\end{align}
and
\begin{equation}\label{eq:CM_dual}
  \cE^D \coloneqq
  \cHom_{\cO _{ X \times Y }} ^{ \codim \cE } ( \cE, \omega _{ X \times Y } )
  =
  \cHom_{\cO _{ X \times Y }} ^{ n } ( \cE, \omega _{ X \times Y } )
\end{equation}
is the \emph{Cohen-Macaulay dual} of
$
 \cE
$
\cite[Definition 1.1.7]{MR2665168}.

%
%

\subsection{Sheaf $\bZ$-algebras}
 \label{sc:shZalg}


A \emph{sheaf
$
 \bZ
$-algebra} over a scheme $X$ is a category $\cC$ enriched over the monoidal category
$
 \lb \shbimod{X}{X}, - \otimes _{ \cO _{ X } } - \rb
$
equipped with a bijection
$
 \bZ \simto \mathop{\mathrm{Obj}} \cC
$.
Concretely, a \emph{sheaf $\bZ$-algebra}
\begin{align}
 \cA =
 \Alg ( \cC )
 =
  \lb
   \lb \cA_{ij} \rb_{i,j \in \bZ},
   \lb \eta_i \rb_{i \in \bZ},
   \lb \mu_{ijk} \rb_{i,j,k \in \bZ}
  \rb
\end{align}
on a scheme $X$
consists of
\begin{itemize}
 \item
sheaf bimodules $\cA_{ij} \in \shbimod{X}{X}$,
 \item
morphisms
$
 \eta_i \colon \cO _{\Delta_{X}} \to \cA_{ii}
$
of sheaf bimodules
called the \emph{units}, and
 \item
morphisms
$
 \mu _{ijk} \colon \cA_{ij} \otimes_{\cO_X} \cA_{jk}
  \to \cA_{ik}
$
called \emph{multiplication maps}
\end{itemize}
such that
\begin{itemize}
 \item
the compositions
\begin{align}
 \cA_{ij} \simto \cA_{ij} \otimes_{\cO_X} \cO_{\Delta_X}
  \xto{\id \otimes \eta_j} \cA_{ij} \otimes_{\cO_X} \cA_{jj}
  \xto{\mu_{ijj}} \cA_{ij}
\end{align}
and
\begin{align}
 \cA_{ij} \simto \cO_{\Delta_X} \otimes_{\cO_X} \cA_{ij}
  \xto{\eta_i \otimes \id} \cA_{ii} \otimes_{\cO_X} \cA_{ij}
  \xto{\mu_{iij}} \cA_{ij}
\end{align}
are the identity morphisms, and
 \item
the diagrams
\begin{align}
\begin{CD}
 \cA_{ij} \otimes_{\cO_X} \cA_{jk} \otimes_{\cO_X} \cA_{kl}
  @>{\mu_{ijk} \otimes \id}>>
 \cA_{ik} \otimes_{\cO_X} \cA_{kl} \\
  @V{\id \otimes \mu_{jkl}}VV @VV{\mu_{ikl}}V \\
 \cA_{ij} \otimes_{\cO_X} \cA_{jl}
  @>{\mu_{ijl}}>>
 \cA_{il}
\end{CD}
\end{align}
are commutative.
\end{itemize}
A (right) \emph{module}
$\cM = \lb (\cM_i)_{i \in \bZ}, (h_{ij})_{i,j \in \bZ} \rb$
over a sheaf $\bZ$-algebra $\cA$ on a scheme $X$
consists of
\begin{itemize}
 \item
$\cO_X$-modules $\cM_i$, and
 \item
morphisms 
$
 h_{ij} \colon \cM_i \otimes_X \cA_{ij} \to \cM_j
$
of $\cO_X$-modules called the \emph{action}
\end{itemize}
such that
\begin{itemize}
 \item
the composition
\begin{align}
 \cM_i \simto \cM_i \otimes_{\cO_X} \cO_{\Delta_X}
  \xto{\id \otimes \eta_i} \cM_i \otimes_{\cO_X} \cA_{ii}
  \xto{h_{ii}} \cM_i
\end{align}
is the identity morphism, and
 \item
the diagrams
\begin{align}
\begin{CD}
 \cM_i \otimes_{\cO_X} \cA_{ij} \otimes_{\cO_X} \cA_{jk}
  @>{h_{ij} \otimes \id}>>
 \cM_j \otimes_{\cO_X} \cA_{jk} \\
  @V{\id \otimes \mu_{ijk}}VV @VV{h_{jk}}V \\
 \cM_i \otimes_{\cO_X} \cA_{ik}
  @>{h _{ i k }}>>
 \cM_k
\end{CD}
\end{align}
are commutative.
\end{itemize}
A \emph{morphism}
$
 f = (f_i)_{i \in \bZ} \colon \cM \to \cN
$
of $\cA$-modules consists of morphisms
$
 f_i \colon \cM_i \to \cN_i
$
of $\cO_X$-modules such that the diagrams
\begin{align}
\begin{CD}
 \cM_i \otimes_{\cO_X} \cA_{ij} @>{f_i \otimes \id}>> \cN_i \otimes_{\cO_X} \cA_{ij} \\
 @V{h^\cM_{ij}}VV @VV{h^\cN_{ij}}V \\
 \cM_j @>{f_j}>> \cN_j
\end{CD}
\end{align}
are commutative.
The category of $\cA$-modules is denoted by $\Gr \cA$.

 An $\cA$-module is \emph{right-bounded}
if $\cM_i \simeq 0$ for $i \gg 0$.
An $\cA$-module is \emph{torsion}
if it is a direct limit of right-bounded objects.
The full subcategory of $\Gr \cA$
consisting of torsion modules is denoted by $\Tor \cA$.
A Grothendieck category is \emph{locally Noetherian}
if it has a small generating family
of Noetherian objects.
In the rest of the paper we will assume that
$
 \cA
$
is right Noetherian
and positively graded in the sense that $\cA_{ij}=0$ for $i>j$, 
so that the category
$
 \Gr \cA
$
is locally Noetherian and
$
 \Tors \cA \subset \Gr \cA
$
is a localizing subcategory. The quotient abelian category is denoted by
$
 \Qgr \cA \coloneqq \Gr \cA / \Tor \cA.
$
The torsion functor
$
 \tau \colon \Gr \cA \to \Tor \cA,
$
the quotient functor
$
 \pi \colon \Gr \cA \to \Qgr \cA,
$
and its right adjoint
$
 \omega \colon \Qgr \cA \to \Gr \cA
$
are defined as in \cite {MR1304753}.
\begin{comment}
  \xc{(If $\cA$ is not noetherian, that is, if $\Gr \cA$ is not locally noetherian, then I do not know if $\Tors \cA$ is a Serre subcategory.  On the other hand, if $\cA$ is noetherian, then $\Tors \cA$ is in fact a localizing subcategory so that the right adjoint $\omega$ exists.)}
\end{comment}
In this case, 
the image of the unit $\eta_i$ will be denoted by $e_i$,
which is an $\cO_{X \times X}$-submodule of $\cA_{ii}$.
It defines the functor
\begin{align}
 (-) \otimes_{\cO_X} e_i \cA
  \colon \Qcoh X \to \Gr \cA.
\end{align}

\subsection{Noncommutative symmetric algebras}
 \label{sc:ncsa}


Let $\cE$ be a locally free sheaf bimodule of rank 2 on a smooth scheme $X$.
We define a sequence $(\cE^{*i})_{i \in \bZ}$
of sheaf bimodules of rank 2 on $X$ inductively by
\begin{align}
 \cE^{*i} \coloneqq
\begin{cases}
 (\cE^{*(i-1)})^{*} & i \ge 1, \\
 \cE & i = 0, \\
 \lstar{(\!^{*(i+1)} \cE)} & i \le -1.
\end{cases}
\end{align}
Let
\begin{align}
 i_n \colon \cO_{\Delta_X} \to \cE^{*n} \otimes_{\cO_X} \cE^{*(n+1)}
\end{align}
be the canonical morphism coming from the adjunction
\begin{align}
 \Hom (\cE^{*n}, \cE^{*n})
  \simeq \Hom \lb \cO_{\Delta_X}, \cE^{*n} \otimes_{\cO_X} \cE^{*(n+1)} \rb
\end{align}
in $\shbimod{X}{X}$,
and $\cQ_n \subset \cE^{*n} \otimes_{\cO_X} \cE^{*(n+1)}$
be the image of $i_n$.
The sheaf bimodule $\cQ_n$ is invertible
since $i_n$ is injective by \cite[Proposition 3.1.10]{MR2958936}.
The \emph{noncommutative symmetric algebra}
is the sheaf $\bZ$-algebra $\bS(\cE)$ on $X$
generated by $\cE^{*i}$ subject to the relations $\cQ_i$.
To be more explicit,
it is a sheaf $\bZ$-algebra with
\begin{align}
 \bS(\cE)_{ij} =
\begin{cases}
 0 & i > j, \\
 \cO _{\Delta _{X}} & i = j, \\
 \cE ^{ * i } & j = i + 1, \\
 \lb \cE ^{ * i } \otimes_{\cO_X} \cdots \otimes_{\cO_X} \cE ^{ * ( j - 1 ) } \rb /
 \cR_{i j} & j > i + 1,
\end{cases}
\end{align}
where
\begin{align}
 \cR_{ i j } \coloneqq \sum_{k=i}^{j-2}
  \cE ^{ * i } \otimes_{\cO_X} \cdots \otimes_{\cO_X} \cE ^{ * ( k - 1 ) } \otimes_{\cO_X}
  \cQ_k \otimes_{\cO_X} \cE ^{ * ( k + 2 ) } \otimes_{ \cO _{ X } }
  \cdots \otimes _{ \cO _{ X } } \cE ^{ * ( j - 1 ) }.
\end{align}

\subsection{Noncommutative $\X$-bundles over commutative schemes}
 \label{sc:ncP1-bundle}

Recall from \cite[Section 3.5]{MR1846352} that
a \emph{quasi-scheme} $X$ is a symbol
with which one associates a Grothendieck category $\Module X$,
and an \emph{enriched quasi-scheme} is a pair
$(X, \cO_X)$
of a quasi-scheme $X$ and
an object $\cO_X$ of $\Module X$.

Let $\cE$ be a locally free sheaf bimodule of rank 2
on a smooth scheme $X$.
The \emph{noncommutative $\bP ^1$-bundle} $\bP(\cE)$
in the sense of \cite{MR2958936}
is the quasi-scheme with
\begin{align}
 \Module \bP(\cE) \coloneqq \Qgr \bS(\cE).
\end{align}
It is naturally enriched by
\begin{align}
 \cO_{\bP(\cE)} \coloneqq \pi \lb \cO_X \otimes_{\cO_X} e_0 \bS(\cE) \rb.
\end{align}
The category $\Gr \bS(\cE)$
and hence $\Qgr \bS(\cE)$
is locally Noetherian
by \cite[Theorem 1.2]{MR2958936}.
The full subcategory of $\Qgr \cA$ consisting of Noetherian objects
will be denoted by $\qgr \cA$.
The exact functor
\begin{align}
 f_n^* \coloneqq \pi \lb (-) \otimes_{\cO_X} e_n \cA \rb
  \colon \Qcoh X \to \Qgr \cA
\end{align}
induces a functor from $\coh X$ to $\qgr \cA$
by \cite[Proposition 2.16]{MR2115370},
which will be denoted by $f_n^*$ again
by an abuse of notation.
The images of $\coh X$ by $f_n^*$
for all $n \in \bZ$ together generates $\Qgr \bS(\cE)$
by \cite[Proposition 2.19]{MR2115370}.
The functor
\begin{align}
  f_{n*} \coloneqq \lb \omega(-) \rb_n
  \colon \Qgr \cA \to \Qcoh X.
\end{align}
is right adjoint to $f_n^*$
by \cite[Lemma 2.15]{MR2115370}.
It is left exact since it has a left adjoint.
Its derived functor,
which exists since $\Qgr \cA$ is a Grothendieck category,
is denoted by $\bR f_{n*}$.
It satisfies
\begin{align}
 \RHom(f_n^* \cF, \cM) \simeq \RHom(\cF, \bR f_{n*} \cM)
\end{align}
by \cite[Lemma 4.2]{MR2333760}.
For any locally free $\cO_X$-module $\cF$,
one has
\begin{align} \label{eq:Mor07.4.4}
 \bR f_{m*}(f_n^* \cF) \simeq
\begin{cases}
 \cF \otimes_{\cO_X} \cA_{n, m} & n \le m, \\
 0 & n = m+1, \\
 \cF \otimes_{\cO_X} \cQ_{n - 2}^* \otimes_{\cO_X} (\cA_{m,n-2})^*[-1] & n \ge m+2
\end{cases}
\end{align}
by \cite[Lemma 4.4]{MR2333760}.
It follows that
\begin{align}
 \bR {f_n}_* \circ f_n^*
  \simeq (-) \otimes_{\cO_X} \cO_{\Delta_X}
  \simeq \id \colon D^b \coh X \to D^b \coh X
\end{align}
for any $n \in \bZ$,
so that $f_n^*$ is full and faithful.

Recall from \cite{MR1039961}
that a Serre functor in a $\Hom$-finite $\bfk$-linear category $\cD$
is an additive equivalence $S \colon \cD \to \cD$
with bi-functorial isomorphisms
$
 \phi_{A,B} \colon \Hom(A, B) \simto \Hom(B, S(A))^*.
$
The category $D^b\qgr \bS(\cE)$ has a Serre functor
by \cite{MR2115370,MR3084432}.

\subsection{Invariance of the module categories under the action of invertible sheaf bimodules}
 \label{sc:twisting}

Let $\cA = \lb \cA_{ij} \rb_{i,j \in \bZ}$ be a sheaf $\bZ$-algebra on $X$,
and $\cT = (\cT_i)_{i \in \bZ}$ be a sequence of invertible objects
in $\shbimod{X}{X}$.
The \emph{twist of $\cA$ by $\cT$} is defined in \cite[Section 3.2]{MR2958936}
as the sheaf
$
 \bZ
$-algebra $\cB \coloneqq \cA_{\cT}$ given by
\begin{align}
 \cB_{ij} \coloneqq \cT_i^{-1} \otimes_{\cO_X} \cA_{ij}
 \otimes_{\cO_X} \cT_j
\end{align}
with the obvious multiplications.
Then the functor
\begin{align}
(\cM_i)_{i\in \bZ} \mapsto 
(\cM_i \otimes_{\cO_X} \cT_i)_{i\in \bZ}
\end{align}
defines an equivalence $\Gr(\cA) \simto \Gr(\cB)$,
which descends to an equivalence $\Qgr(\cA) \simto \Qgr(\cB)$.

Let $\cE$ be a locally free sheaf bimodule.
The definition
\begin{align}
 \Hom(\cM \otimes_{\cO_X} \cE, \cN)
  \simeq \Hom(\cM, \cN \otimes_{\cO_X} \cE^*)
\end{align}
of the adjunction implies
\begin{align}
 (\cT_0^{-1} \otimes_{\cO_X} \cE \otimes_{\cO_X} \cT_1)^*
  \simeq \cT_1^{-1} \otimes_{\cO_X} \cE^* \otimes_{\cO_X} \cT_0
\end{align}
for any pair $(\cT_0, \cT_1)$ of invertible sheaf bimodules.
It follows that
\begin{align}
 (\cT_0^{-1} \otimes_{\cO_X} \cE \otimes_{\cO_X} \cT_1)^{*m} \simeq
\begin{cases}
 \cT_0^{-1} \otimes_{\cO_X} \cE^{*m} \otimes_{\cO_X} \cT_1 & m \colon \mbox{even}, \\
 \cT_1^{-1} \otimes_{\cO_X} \cE^{*m} \otimes_{\cO_X} \cT_0 & m \colon \mbox{odd},
\end{cases}
\end{align}
so that
\begin{align}
 \bS(\cT_0^{-1} \otimes_{\cO_X} \cE \otimes_{\cO_X} \cT_1)
  \simeq \bS(\cE)_{\cT},
\end{align}
where $\cT = (\cT_i)_{i \in \bZ}$ is defined by
\begin{align}
 \cT_i \coloneqq
\begin{cases}
 \cT_0 & i : \text{even}, \\
 \cT_1 & i : \text{odd} \\
\end{cases}
\end{align}
(see also \cite[p. 136]{MR2513212}).

\begin{example}\label{eg:invertible_sheaf_bimodules}
\begin{enumerate}
\item\label{it:automorphism_as_invertible_sheaf_bimodules}
An example of an invertible sheaf bimodule
is given by the structure sheaf $\cO_{\Gamma_g}$
of the graph
\begin{align}
 \Gamma_g \coloneqq \lc (x, g x) \in X \times X \relmid x \in X
 \rc
\end{align}
of $g \in \Aut X$.
One can easily check
$
 (-) \otimes _{ \cO _{ X } } \iota _{ * } \cO_{\Gamma_g} (-)
 \simeq
 g _{ * }
$.

For any pair
$
 \lb g, h \rb
$
of automorphisms on
$
 X
$,
consider the pair of invertible sheaf bimodules
$
 (\cT_0, \cT_1) = \lb \cO _{ \Gamma _{ g } }, \cO _{ \Gamma _{ h } } \rb
$.
One can easily check
\begin{equation}
 \cT_0^{-1} \otimes_{\cO_X} \cE \otimes_{\cO_X} \cT_1
 \simeq
 \lb g \times h \rb _{ * } \cE.
\end{equation}
By the arguments above, this sheaf bimodule gives rise to the equivalent category as
$
 \cE
$
does. In short,
$
 \Aut X \times \Aut X
$
acts on the space of sheaf bimodules preserving
the equivalence classes of the resulting abelian categories $\qgr \bS(\cE)$.

\item\label{it:line_bundle_as_invertible_sheaf_bimodules}
Another example is given by the line bundle
$
 \Delta _{ * } L
$
on the diagonal of
$
 X \times X
$.
Consider the pair of invertible sheaf bimodules
$
 (\cT_0, \cT_1)
 =
 \lb \Delta _{ * } L, \Delta _{ * } M  \rb
$.
By standard arguments one can verify
\begin{align}
 \cT_0^{-1} \otimes_{\cO_X} \cE \otimes_{\cO_X} \cT_1
 \simeq
 \cE \otimes _{ \cO _{ X \times X } } \lb L ^{ - 1 } \boxtimes M \rb.
\end{align}
In particular, when
$
 X = \X
$
and
$
 \lb L, M \rb
 =
 \lb \cO _{ \X } ( a ), \cO _{ \X } ( b ) \rb
$,
then
\begin{equation}
 \cT_0^{-1} \otimes_{ \cO _{ \X } } \cE \otimes_{ \cO _{ \X } } \cT_1
 \simeq
 \cE \otimes _{ \cO _{ \X \times \X } } \cO _{ \X \times \X } \lb - a, b \rb.
\end{equation}

Note that the line bundle
$
 \cL \coloneqq u ^{ * } \cO _{ \X } ( - 1 ) \otimes v ^{ * } \cO _{ \X } ( 1 )
$
is non-trivial of degree zero on
$
 W
$, unless
$
 \Wred
$
is a divisor of bidegree
$
 ( 1, 1 )
$.
This implies that there exists the action
$
 \cE = \iota _{ * } \cU \mapsto \iota _{ * } \lb \cU \otimes \cL \rb
$
on the space of sheaf bimodules of fixed degree.
This is an effective action of
$
 \bZ
$
if
$
 \cL \in \Pic ^{ 0 } \lb W \rb
$
is not a torsion point.
\end{enumerate}
\end{example}
Combining these examples with the isomorphism
$
 \Auteq \lb \coh X \rb \simeq \Pic \lb X \rb \rtimes \Aut X
$,
it follows that the natural action of the group
$
 \Auteq \lb \coh X \rb \times \Auteq \lb \coh X \rb
$
on sheaf bimodules preserves the equivalence classes of the associated categories
$
 \Gr \bS ( - )
$
and
$
 \Qgr \bS ( - )
$.

\section{Explicit classification of sheaf bimodules on $\X$}
 \label{sc:sbc}

In this section,
we give an explicit classification of locally free sheaf bimodules of rank 2 on $\X$.
The first part of this section has overlap with \cite[Chapter 4]{thesis:Patrick}, especially with \cite[Theorem 4.5]{thesis:Patrick}. A difference is in that our argument is more abstract and based on the Serre's conditions.

Also we explicitly compute the values of
$
 a, b, a ', b '
$
which are defined by
$
 \cO _{ \X } \otimes \cE
 =
 v _{ * } \cU
 =
 \cO _{ \X } ( a ) \oplus \cO _{ \X } ( b )
$
and
$
 \cO _{ \X } ( - 1 ) \otimes \cE
 =
 \cO _{ \X } ( a ' ) \oplus \cO _{ \X } ( b ' )
$.
This will later be used to check the strongness of certain exceptional collections on
$
 D ^{ b } \qgr \bS \lb \cE \rb
$
and compute its endomorphism algebra.

Let $\cE$ be a locally free sheaf bimodule
on smooth projective schemes $X$ and $Y$
of the same dimension $n$.
Since
$
 u _{ * } \cU
$
is locally free and
$
 X
$
is smooth over a field,
$
 u _{ * } \cU
$
is maximally Cohen-Macaulay.
Since
$
 u
$
is finite, this implies that
$
 \cU
$
is also maximally Cohen-Macaulay over
$
 W
$
(this argument is quoted from the proof of \cite[Proposition 3.1.6]{MR2958936})
and
$
 \cE = \iota _{ * } \cU
$
is a (non-maximal) Cohen-Macaulay module
over $X \times Y$.
Hence by \cite[Proposition 1.1.10]{MR2665168},
the sheaf $\cE$ is \emph{pure},
i.e., all associated points of $\cE$ have the same dimension;
equivalently, non-trivial subsheaves of
$
 \cE
$
have the same dimension as
$
 \cE
$
(see \cite[p.~3]{MR2665168}).
It follows that
$W$ is an equi-$n$-dimensional scheme and
the restriction of
$
 u \colon W \to X
$
to any irreducible component of $W$
is dominant and finite.
Note that here and below,
any claim on $u$ also holds for $v$
by symmetry.

\pref{lm:S1_and_purity} below shows
that the scheme $W$ satisfies
the $S_1$ condition of Serre.

\begin{lemma}\label{lm:S1_and_purity}
Let $A$ be a Noetherian ring and $M$ a finitely generated $A$-module which satisfies the
$
 S _{ 1 }
$
condition and
$
 \ann _{ A } \lb M \rb = 0
$.
Then $A$ also satisfies the
$
 S _{ 1 }
$
condition.
In particular, $A$ is pure as a module over $A$.
\end{lemma}

\begin{proof}
All the properties discussed here are local, so we may assume without loss of generality that
$
 A = ( A, \frakm )
$
is a local ring. Since $M$ satisfies the
$
 S _{ 1 }
$
condition, there exists a non-zero element $a \in \frakm$ such that
$
 M \xto[]{a \cdot} M
$
is injective. Then the assumptions imply that
$
 A \xto[]{a \cdot} A
$
is also injective. Thus we obtain the first claim.

The second claim simply follows from the fact that the condition
$
 S _{ 1 }
$
is equivalent to the purity;
this is essentially stated in
\cite[Proposition 1.1.10]{MR2665168} modulo the simple observation that
a coherent sheaf on a smooth projective scheme whose support is of codimension
$
 c
$
satisfies the
$
 S _{ k, c }
$
condition precisely when it satisfies
$
 S _{ k } = S _{ k, 0 }
$
on its support.
\begin{comment}
In fact, suppose that there exists
$
 0 \neq b \in B
$
such that
$
 b \varphi ( a ) = 0
$.
By assumption
$
 0 \neq b M \subset M
$,
but then
$
 a \cdot ( b M )
 =
 ( \varphi ( a ) b ) M
 =
 0
$
and it contradicts the regularity of
$
 a
$.
\end{comment}
\end{proof}

Since the reducedness of a scheme is equivalent to
$S_1 + R_0$
($=$ 
regular in codimension $0$),
we see that $W$ is reduced
if and only if it is reduced at the generic point of any irreducible component.

\begin{proposition}\label{pr:classification_of_sheaf_bimodules}
If the rank of
$
 \cE
$
is $2$,
then one of the following cases occur.
\begin{enumerate}
\item
$
 W
$
is an irreducible and reduced ($\iff$ integral) scheme.
Then either
\begin{enumerate}
 \item
 $
  u
 $
 is birational. In this case, since
 $
  X
 $
 is normal,
 $
  u
 $
 is an isomorphism. In particular,
 $
  \cU
 $
 is a locally free sheaf of rank $2$ on
 $
  X \simeq W
 $, or

 \item
 $
  u
 $
 is of degree $2$ and $\cU$ is a pure sheaf of rank $1$ on $W$.
 Then by the assumption
 $ \rank \lb v _{ * } \cU \rb = 2 $,
 the degree of $v$ has to be $2$ as well.
\end{enumerate}

 \item
$
 W
$
is irreducible and not reduced.
In this case, by the arguments above,
$
 W
$
is not reduced at the generic point. Therefore
$
 \Wred \to X
$
and
$
 \Wred \to Y
$
are both birational and finite, and hence are isomorphisms.
The sheaf
$
 \cU
$
is pure, and isomorphic to
$
 \cO _{ W }
$
on a non-empty Zariski open subset of $W$.

\item
$
 W
$
is not irreducible.
In this case, $W$ is reduced and admits exactly 2 irreducible components.
$u$ sends each component of $W$ onto $X$ (so does $v$ onto $Y$).
The sheaf
$
 \cU
$
is pure and of rank $1$ on each irreducible component.
\end{enumerate}
\end{proposition}

Let us assume that $X$ and $Y$ are smooth projective curves.
As shown in \pref{lm:S1_and_purity} there exists no embedded point in
$
 W
$.
We also checked that any irreducible component of
$
 W
$
dominates the generic point of
$
 X
$.
Therefore
\cite[Chapter III, Proposition 9.7]{MR0463157}
implies:

%

\begin{corollary}
$
 u \colon W \to X
$
and
$
 v \colon W \to Y
$
are flat.
\end{corollary}

One can also show:

\begin{corollary} \label{cr:Cartier}
$W$ is isomorphic to a Cartier divisor of $X \times Y$.
\end{corollary}

\begin{proof}
Consider the standard short exact sequence
\begin{equation}\label{eq:Cartier}
 0 \to I _{ W } \to \cO _{ X \times Y } \to \iota _{ * } \cO _{ W } \to 0.
\end{equation}
For any coherent sheaf
$
 E
$
on
$
 X \times Y
$,
there exists the canonical map to the \emph{double dual}
\begin{equation}
 \theta _{ E } \colon E \to E ^{ D D },
\end{equation}
where
$
 ( - ) ^{ D }
$
is the Cohen-Macaulay dual as defined in \eqref{eq:CM_dual}.
Then
$
 E
$
is said to be \emph{reflexive} if
$
 \theta _{ E }
$
is an isomorphism. Note that
$
 W
$
is a Cartier divisor if and only if
$
 I _{ W }
$
is an invertible sheaf.
This is, in turn, equivalent to the condition that
$
 I _{ W }
$
is reflexive in the above sense.

Apply the functor
$
 \cExt ^{ i } _{ X \times Y } \lb -, \omega _{ X \times Y } \rb
$
to
\pref{eq:Cartier}
to obtain the following long exact sequence:
\begin{equation}
 \begin{split}
 \cExt ^{ 1 } _{ X \times Y } \lb \iota _{ * } \cO _{ W }, \omega _{ X \times Y } \rb
 \to
 \cExt ^{ 1 } _{ X \times Y } \lb \cO _{ X \times Y }, \omega _{ X \times Y } \rb
 = 0
 \to
 \cExt ^{ 1 } _{ X \times Y } \lb I _{ W }, \omega _{ X \times Y } \rb\\
 \to
 \cExt ^{ 2 } _{ X \times Y } \lb \iota _{ * } \cO _{ W }, \omega _{ X \times Y } \rb
 \to
 \cExt ^{ 2 } _{ X \times Y } \lb \cO _{ X \times Y }, \omega _{ X \times Y } \rb
 =
 0
 \to
 \cExt ^{ 2 } _{ X \times Y } \lb I _{ W }, \omega _{ X \times Y } \rb\\
 \to
 0.
 \end{split}
\end{equation}
Since
$
 \iota _{ * } \cO _{ W }
$
satisfies
$
 S _{ 1, 1 }
$,
$
  \cExt ^{ 2 } _{ X \times Y } \lb \iota _{ * } \cO _{ W }, \omega _{ X \times Y } \rb
  =
  0
$
by
\cite[Proposition 1.1.10]{MR2665168}.
Hence
$
 \cExt ^{ q } _{ X \times Y } \lb I _{ W }, \omega _{ X \times Y } \rb
 =
 0 
$
for
$
 q = 1, 2
$,
which in turn implies
$
 I _{ W }
$
is reflexive again by
\cite[Proposition 1.1.10]{MR2665168}.
\end{proof}

%

%
%



From now on,
we will restrict ourselves to the case when
$
 X = Y = \X
$
and
$
 \cE
$
is a locally free sheaf bimodule of rank 2.
As we proved in \pref{pr:classification_of_sheaf_bimodules},
such a sheaf bimodule is either of the following two types
(we borrow the labels from \cite[Definition 6.2.1]{MR2958936}):
\begin{enumerate}[(I)]
 \item
 \label{it:type_1_1}
$W$ is a divisor of bidegree $(1,1)$,
and $\cU$ is a locally free sheaf on $W$ of rank 2.

 \item
 \label{it:type_2_2}
$W$ is a divisor of bidegree $(2,2)$,
and $\cU$ is a pure sheaf on $W$ of rank 1.
\end{enumerate}
Since
\begin{itemize}
 \item
a divisor of bidegree $(1,1)$ in $\X \times \X$ is isomorphic to
$
 \X
$
and transferred to the diagonal by an automorphism of $\X \times \X$,
 \item
a pure sheaf on a smooth curve is locally free, and
 \item
a locally free sheaf of rank 2 on $\X$ is isomorphic to
$\cO_\X(a) \oplus \cO_\X(b)$
for some $a, b \in \bZ$
by Birkhoff--Grothendieck theorem,
\end{itemize}
a noncommutative Hirzebruch surface associated to a sheaf bimodule of type \eqref{it:type_1_1}
is equivalent to a commutative Hirzebruch surface by \pref{eg:invertible_sheaf_bimodules} \eqref{it:automorphism_as_invertible_sheaf_bimodules}.
Hence in the rest of this section we may and will consider sheaf bimodules of type \eqref{it:type_2_2}.
Since
$
 \cO_\X \otimes_{\cO_\X} \cE \simeq v _{ * } \cU
$
is a locally free sheaf of rank 2 on $\X$,
there exist integers $a \le b \in \bZ$ such that
\begin{align}\label{eq:decomposition_of_v*U}
 v _{ * } \cU \simeq \cO_\X(a) \oplus \cO_\X(b).
\end{align}

Since $W$ is a divisor of bidegree $(2,2)$ in $\X \times \X$,
the dualizing sheaf of $W$ is trivial, so that the Serre duality theorem implies
\begin{align}\label{eq:Serre_duality_for_embedded_curves}
 h^1(\cU) = h^0(\cU^\vee)
\end{align}
for any locally free $\cO_W$-module $\cU$.
Also, it follows from the Riemann-Roch theorem
\cite[Chapter II, Theorem (3.1)]{MR2030225} that
\begin{align}\label{eq:Riemann-Roch_for_embedded_curves}
 \chi \lb \cU \rb = h^0(\cU) - h^1(\cU) = \rank \cU \cdot \deg \cU,
\end{align}
where
$
 \deg \cU
 \coloneqq
 m \cdot \deg \lb \cU | _{ \Wred } \rb
$
when
$
 W = m \Wred
$
as Weil divisors on
$
 \X \times \X
$.
As an immediate corollary, we easily obtain the formula
\begin{align}\label{eq:consequence_from_the_euler_number}
 \deg \cU
 =
 \chi \lb \cU \rb
 =
 \chi \lb v _{ * } \cU \rb
 =
 a + b + 2
\end{align}

\subsection{Non-reduced $W$}
Suppose that $W$ is not reduced.
Then the reduced subscheme $\Wred$ is a divisor of bidegree $(1,1)$,
which we may and will assume
$
 \Wred = \Delta _{ \X }
$
without changing the abelian category by \pref{eg:invertible_sheaf_bimodules} \eqref{it:automorphism_as_invertible_sheaf_bimodules}.

\begin{theorem} \label{th:sheaf_bimodule_non-reduced}
The sheaf $\cU$ on $W$ sits in an exact sequence of the following form
\begin{align}\label{eq:sheaf_bimodule_non-reduced}
 0
 \to
 \cU
 \to
 \cL
 \xto[]{c}
 \cL \otimes \cO _{ D } \simeq \cO _{ D }
 \to
 0,
\end{align}
where
\begin{itemize}
\item
$
 \cL
$
is an invertible sheaf on
$
 W
$,

\item
$
 D = \sum _{ i = 1 } ^{ N } n _{ i } x _{ i } \subset \Wred \subset W
$
is a $0$-dimensional closed subscheme of $W$, where
$
 N \in \bZ _{ \ge 0 }
$,
$
 x _{ 1 }, \dots, x _{ N } \in \Wred
$
are distinct closed points, and
$
 n _{ 1 }, \dots, n _{ N } \in \bZ _{ > 0 }
$.

\item
$c$ is the restriction morphism.
\end{itemize}
Conversely, any such sheaf $\cU$ on $W$, regarded as a sheaf on $\X \times \X$, is a locally free sheaf bimodule of rank 2 whose support coincides with $W$.
\end{theorem}

By an abuse of notation we let
$
 \cO _{ D }
$
denote both the structure sheaf of $D$ and its pushforwards to ambient schemes, depending on the context.
\pref{th:sheaf_bimodule_non-reduced}
is an immediate corollary of
\pref{lm:local_str_nonred} below on the local structure of $\cU$.

\begin{lemma}[{$=$ \cite[Proposition 4.1]{MR877011}}] \label{lm:local_str_nonred}
Consider the local $\bfk$-algebra
$
 A = \bfk [t, \varepsilon ] / ( \varepsilon ^{ 2 } ) _{ ( t ) }
$
and let
$
 M
$
be an MCM $A$-module such that
$
 \ann _{ A } ( M ) = 0
$.
Then there exists a uniquely determined non-negative integer
$
 n
$
such that $M$ is isomorphic to the ideal
$
 ( t^n, \varepsilon )
$
as
$
 A
$-modules.
\end{lemma}

Let
$
 x \in W
$
be a closed point at which
$
 \cU
$
is not locally free. By taking the module $M$ of \pref{lm:local_str_nonred} to be the stalk
$
 \cU _{ x }
$
and
$
 A
$
to be
$
 \cO _{ W, x }
$,
one can find an embedding
$
 \cU _{ x }
 \hookrightarrow
 \cO _{ W, x }
$
as an ideal. The invertible sheaf
$
 \cL
$
is obtained from
$
 \cU
$
by locally replacing
$
 \cU _{ x }
$
with the over module
$
 \cO _{ W, x }
$
at each such point $x$.

Tensoring
\eqref{eq:sheaf_bimodule_non-reduced}
with
$
 \cL ^{ - 1 }
$, we obtain
\begin{align}
 0 \to \cU \otimes \cL ^{ - 1 } \to \cO _{ W } \to \cO _{ D } \to 0.
\end{align}
On the other hand, since
$
 I_{ \Wred / W }
 \simeq
 i ^{ * } \cO_{\X \times \X}( - 1, - 1 )
 \simeq \cO_{\Wred}( - 2 ),
$
where
$
 i \colon \Wred \to W
$
is the canonical inclusion, we obtain a morphism of exact sequences as follows. The rightmost vertical map corresponds to the closed immersion
$
 D \hookrightarrow \Wred
$.
\begin{align}
 \xymatrix{
 0 \ar[r] & i _{ * } \cO _{ \Wred } ( - 2 ) \ar[r] \ar[d] & \cO _{ W } \ar[r] \ar@{=}[d] & i _{ * } \cO _{ \Wred } \ar[r] \ar[d] & 0\\
 0 \ar[r] & \cU \otimes \cL ^{ - 1 } \ar[r] & \cO _{ W } \ar[r] & \cO _{ D } \ar[r] & 0
 }
\end{align}
Applying the snake lemma, we obtain the short exact sequence
\begin{align}\label{eq:decomposition_of_U_L_-1}
 0 \to i _{ * } \cO _{ \Wred } ( - 2 ) \to \cU \otimes \cL ^{ - 1 } \to i _{ * } \cO _{ \Wred } ( - D ) \to 0.
\end{align}
Applying
$
 - \otimes \cL
$
we obtain the exact sequence
\begin{align}\label{eq:decomposition_of_U}
 0
 \to
 \cL \otimes _{ W } i _{ * } \cO _{ \Wred } ( - 2 )
 \to
 \cU
 \to
 \cL \otimes _{ W } i _{ * } \cO _{ \Wred } ( - D )
 \to
 0,
\end{align}
which locally around the point
$
 x _{ i }
$
is isomorphic to the exact sequence
\begin{align} \label{eq:local_nonred_ES}
 0 \to (\varepsilon) \to ( t ^{ n _{ i } }, \varepsilon ) \to ( t ^{ n _{ i } }, \varepsilon ) / ( \varepsilon ) \to 0
\end{align}
of 
$A$-modules. From this local description, one can conclude that the support of the divisor $D$ in \pref{th:sheaf_bimodule_non-reduced}
coincides with the non-locally-free locus of $\cU$.

Applying
$
 v _{ * }
$
to
\eqref{eq:decomposition_of_U}, we obtain the following exact sequence on $\X$.
\begin{align} \label{eq:L_U_L}
 0
 \to
 i ^{ * } \cL ( - 2 )
 \to
 v _{ * } \cU
 \to
 i ^{ * } \cL ( - D )
 \to
 0.
\end{align}
If
$
 \deg D > 0,
$
then the sequence \pref{eq:L_U_L} splits,
and one has
$
 v _{ * } \cU
 \simeq
 i^* \cL ( - 2 ) \oplus i^* \cL ( - \deg D )
$.
Therefore when
$
 \deg D > 0
$,
the deformations of $\cU$ correspond to the deformations of the pair
$
 \lb \cL, D \rb
$,
so that we obtain a
$
 \lb \dim \Pic ^{ 0 } W + \deg D = \rb \lb \deg D + 1 \rb
$-dimensional family of sheaf bimodules $\cU$.
Taking into account the action of $\Aut(\X \times \X)$ on the space of sheaf bimodules,
which preserves the equivalence classes of the associated categories $\Qgr \bS(\cE)$,
when
$
 \deg D \ge 2
$,
we obtain a
$
 \lb (\deg D + 1) - 3 = \rb ( \deg D - 2 )
$-dimensional family of noncommutative Hirzebruch surfaces;
note that the anti-diagonal subgroup of $\Aut(\X \times \X)$ has been already used to translate $\Wred$ to the diagonal.

%

To describe invertible sheaves on $W$, consider the defining ideal
$
 \cI \subset \cO_W
$
of $\Wred$. It satisfied
$
 \cI ^{ 2 } = 0
$,
and since
$
 \Wred
$
is a divisor of bidegree
$
 ( 1, 1 )
$, it follows that
$
 \cI \simeq i_* \cO_{\Wred} ( - 2 )
$. Note that
$
 \cN _{ \Wred / W } \simeq \lb \cI / \cI ^{ 2 } \rb ^{ \vee } = \cI ^{ \vee } \simeq \cO_{\Wred} ( 2 )
$.
It then follows that
\begin{align}
 W \simeq V ( \cJ ^{ 2 } ) \subset \cN_{ \Wred / W } = \Spec_{\Wred} \Sym_{\cO_{\Wred}} ^{ \bullet } \cO_{\Wred}(-2),
\end{align}
where
$
 \cJ = \Sym_{\cO_{\Wred}} ^{ > 0 } \cO_{\Wred} ( - 2 )
$
is the ideal sheaf of the 0-section of the normal bundle
$
 \cN_{ \Wred / W }
$.
From this isomorphism, we see that the scheme $W$ is 
obtained by gluing
$
 U_1 \coloneqq \Spec \bfk[z,u]/(u^2)
$
and
$
 U_2 \coloneqq \Spec \bfk[w,v]/(v^2)
$
along
$
 U_{12} \coloneqq
  \Spec \bfk[z,w,u,v]/(u^2,v^2,zw-1, u-z^2v)
$.

Consider the short exact sequence
\begin{equation}
 1 \to \cI \xto{e} \cO_W^\times \to \cO _{\Wred}^\times \to 1
\end{equation}
of sheaves of abelian groups
on the topological space underlying $W$, where
$e$ is the map defined by
$
 x \mapsto x + 1.
$
By taking the long exact sequence, we obtain the exact sequence
\begin{equation}
 1 \to
 H^1 \lb \Wred, \cO_{\Wred}(-2) \rb
 \xto[]{ H ^{ 1 } ( e ) }
 \Pic ( W )
 \xto[]{\deg \lb \bullet | _{ \Wred } \rb}
 \Pic ( \Wred ) \simeq \bZ
 \to
 1
\end{equation}
of abelian groups. In particular,
$
 \deg \cU = 2 \deg \lb \cU | _{ \Wred } \rb
$
is always an even integer.

Via the explicit description of $W$ given above, the \v{C}ech complex for
$
 \cO _{ \Wred } ( - 2 )
$
with respect to the affine cover
$
 \Wred = U _{ 1, \mathrm{red} } \cup U _{ 2, \mathrm{red} }
$
is described as follows.
\begin{align}
 \bfk [ z ] \oplus \bfk [ w ] \xto[]{d} \bfk [ z, w ] / ( z w - 1 ); \quad \quad
 \lb f ( z ), g ( w ) \rb \xmapsto[]{d} f ( z ) - \frac{1}{ z ^{ 2 } } g \lb \frac{ 1 }{ z } \rb
\end{align}
Thus we see that
$
 \Hv ^1 \lb \Wred, \cO_{\Wred}(-2) \rb = \bfk \ld w \rd
$.

Under the morphism
$
 \Hv ^1 \lb \Wred, \cO_{\Wred}(-2) \rb
 \xto[]{ H ^{ 1 } ( e ) }
 \Hv ^1 \lb W, \cO_{W}(-2) \rb
$,
the element
$
 a [ w ]
$
is mapped to the class represented by the cocycle
$
 1 + a u w = 1 + a v z \in \cO ^{ \times } \lb U _{ 1 2 } \rb
$.
Therefore it follows that any line bundle on $W$ of degree zero is given as the line bundle $\cL_a$
obtained by gluing
$\cO_{U_1}$ and $\cO_{U_2}$
by $1 + a z v \in \cO_{W}^\times(U_{12})$ for $a \in \bfk$.

The \v{C}ech complex for $\cL_a$ is given by
\begin{align}
 \bfk[z,u]/(u^2) \oplus \bfk[w,v]/(v^2) &\to \bfk[z,w,u,v]/(u^2,v^2,zw-1, u-z^2v), \\
 (f(z,u),g(w,v)) &\mapsto (1+a z v) f(z,u)-g(w,v),
\end{align}
which is acyclic if and only if $a \ne 0$.
More precisely, it follows that
\begin{align}\label{eq:cohomology_of_La}
 \dim H^i(\cL_a)
  =
\begin{cases}
 1 & a = 0, \\
 0 & a \ne 0,
\end{cases}
\end{align}
for $i=0,1$.
Since
$
 v
$
is an affine morphism, we have
\begin{align}
  H^i(\cL_a) \simeq H^i( v_*(\cL_a)).
\end{align}
Combining this with \eqref{eq:cohomology_of_La}, we see that
\begin{align}
 v _{ * } \cL_a
  \simeq
\begin{cases}
 \cO_{ \X } \oplus \cO_{ \X }(-2) & a = 0, \\
 \cO_{ \X }(-1) \oplus \cO_{ \X }(-1) & a \ne 0.
\end{cases}
\end{align}
Since the morphisms $u$ and $v$ are isomorphic, one also has
\begin{align}
 \cO_{ \X }(-1) \otimes \cE
  = v_*(v^*(\cO _{ \X }(-1)) \otimes \cL _{ a })
  = \cO _{ \X } ( - 1 ) \otimes v_* \cL _{ a }.
\end{align}

Consider in general
$
 \cU
$
with
$
 \cL | _{ \Wred } \simeq \cO _{ \Wred } ( q )
$
and
$
 \deg D > 0
$.
Then
$
 v _{ * } \cU
 \simeq
 \cO _{ \X } ( q - 2 )
 \oplus
 \cO _{ \X } ( q - \deg D )
$
as we observed after \eqref{eq:L_U_L}.

Summing up, we obtain the following conclusion.
\begin{lemma}\label{lm:non_reduced_W}
Suppose that
$
 W
$
is not reduced.

\begin{enumerate}
\item
When
$
 \cU
$
is invertible, then
\begin{align}
 ( a, b ), b - a
 =
 \begin{cases}
 \lb \frac{\deg \cU }{ 2 } - 2, \frac{\deg \cU }{ 2 } \rb, 2
 & \cU \simeq v ^{ * } \cO _{ \X } \lb k \rb \ \lb \exists k \in \bZ \rb, \\
 \lb \frac{\deg \cU }{ 2 } - 1, \frac{\deg \cU }{ 2 } - 1 \rb, 0
 & \cU \not \simeq v ^{ * } \cO _{ \X } \lb k \rb \ \lb \forall k \in \bZ \rb.
 \end{cases} 
\end{align}

\item
When
$
 \cU
$
is not invertible, so that
$
 \deg D > 0
$,
\begin{align}
 ( a, b ), b - a
 =
\begin{cases}
 \lb \frac{ \chi (\cU) - 3 }
 { 2 }, \frac{ \chi (\cU) -1 } 
 { 2 } 
 \rb,
 1 & \deg D = 1\\
 \lb \frac{ \chi (\cU) - \deg D} 
 { 2 },  
 \frac{ \chi (\cU)+\deg D} 
 { 2 } - 2 \rb,
 \deg D - 2 & \deg D \ge 2.
\end{cases}
\end{align}
\end{enumerate}
\end{lemma}


%
%
%
\subsection{Reduced $W$}

In this section we assume that $W$ is a reduced divisor of bidegree $(2,2)$.
Either by studying the branched double covers
$u \colon W \to \X$
or by sending $W$ 
by a birational map $\X \times \X \dashrightarrow \bP^2$
and using the classification of cubic curves in $\bP^2$,
one can show that there are five possibilities for $W$ as follows. In the rest of this paper, for convenience, we will use the Kodaira's symbol for singular elliptic fibers to describe the type of $W$. Note that in the list below, except the case $I _{ 0 }$, the type uniquely determines the isomorphism class of $W$.
\begin{enumerate}[(i)]
 \item
$W$ is an elliptic curve ($
 I _{ 0 }
$).

 \item
$W$ is an irreducible nodal rational curve ($I _{ 1 }$).

 \item
$W$ is the union $W_1 \cup W_2$ of two smooth rational curves
intersecting at two points ($I _{ 2 }$).

 \item
$W$ is a cuspidal rational curve ($II$).

 \item
$W$ is the union $W_1 \cup W_2$ of two smooth rational curves
intersecting at one point with multiplicity two ($III$).
\end{enumerate}

\begin{lemma} \label{lm:irreducible_W_invertible_U}
Suppose that
$
 W
$
is irreducible and $\cU$ is invertible. 
Then one has
\begin{align}
 ( a, b ), b - a
 =
 \begin{cases}
 \lb \frac{\deg \cU }{ 2 } - 2, \frac{\deg \cU }{ 2 } \rb, 2
 & \cU \simeq v ^{ * } \cO _{ \X } \lb k \rb \ \lb \exists k \in \bZ \rb, \\
 \lb \frac{\deg \cU }{ 2 } - 1, \frac{\deg \cU }{ 2 } - 1 \rb, 0
 & \cU \not \simeq v ^{ * } \cO _{ \X } \lb k \rb \ \lb \forall k \in \bZ \rb \mbox{ and } \deg \cU \equiv 0 \mod 2, \\
 \lb \frac{\deg \cU - 1 }{ 2 } - 1, \frac{\deg \cU - 1 }{ 2 } \rb, 1 & \deg \cU \equiv 1 \mod 2,
 \end{cases} 
\end{align}
where the integers
$
 a, b
$
are those in \eqref{eq:decomposition_of_v*U}.

%
\begin{comment}
\begin{align}
 \cO_\X \otimes_{\cO_\X} \cE \simeq
\begin{cases}
 \cO_\X \oplus \cO_\X(-2) & \cU \simeq \cO_W, \\
 \cO_\X(-1) \oplus \cO_\X(-1) & \cU \not \simeq \cO_W \text{ and } \deg \cU = 0, \\
 \cO_\X \oplus \cO_\X(-1) & \deg \cU = 1.
\end{cases}
\end{align}
\end{comment}
\end{lemma}

\begin{proof}
The first claim follows from the following computations
\begin{align}
 \chi _{ W } \lb \cU \rb = \chi _{ \X } \lb v _{ * } \cU \rb
 = \chi _{ \X } \lb \cO _{ \X } \lb a \rb \oplus \cO _{ \X } \lb b \rb \rb
 =
 \lb a + 1 \rb + \lb b + 1 \rb.
\end{align}
For the first equality we used that
$
 v
$
is an affine morphism.

Let us show the second claim. By replacing
$
 \cU
$
with
$
 \cU \otimes v ^{ * } \cO _{ \X } \lb k \rb
$
for some integer
$
 k
$,
we may and will assume
$
 \deg \cU = 0
$
or
$
 1
$, without changing the value
$
 b - a
$.

If $\deg \cU = 0$,
\eqref{eq:consequence_from_the_euler_number} implies
\begin{align}\label{eq:proof_of_lemma_9_4}
 a + b + 2 = \chi \lb \cU \rb = 0
\end{align}
On the other hand we have the following dichotomy
\begin{align}
 h^0(v_* \cU) =  h^0(\cU) =
\begin{cases}
 1 & \cU \simeq \cO_W, \\
 0 & \text{otherwise}.
\end{cases}
\end{align}
In the first case it follows that $a < b = 0$,
and in the second case $a \le b \le -1$. Combining this with
\eqref{eq:proof_of_lemma_9_4},
in each of the two cases we conclude
$
 \lb a, b \rb = \lb - 2, 0 \rb
$
and
$
 \lb - 1, - 1 \rb
$,
respectively.
%
%

If $\deg \cU = 1$,
then one has
\begin{align}
 h^1(\cU) = h^0(\cU^\vee) = 0,
\end{align}
which together with \eqref{eq:Riemann-Roch_for_embedded_curves} implies
\begin{align}
 h^0(\cU) = \chi \lb \cU \rb = 1.
\end{align}
Combining this with
\eqref{eq:consequence_from_the_euler_number} as before,
we conclude that
$
 \lb a, b \rb = \lb - 1, 0 \rb
$.
Thus we finish the proof of \pref{lm:irreducible_W_invertible_U}.
\end{proof}

\begin{lemma}\label{lm:non_irreducible_W_invertible_U}
Suppose that
$
 W
$
is \emph{not} irreducible (i.e.,
$
 I _{ 2 }
$
or
$
 III
$), so that it is a sum of two distinct effective divisors of bidegree
$
 ( 1, 1 )
$
and
$
 \cU
$
is an invertible sheaf of bidegree
$
 \lb p, q \rb
$
on
$
 W
$
with
$
 p \le q
$.
Then
\begin{align}
 ( a, b ), b - a
 =
 \begin{cases}
 ( p - 2, p ), 2 & \mbox{if} \  q - p = 0 \mbox{ and } \cU \simeq v ^{ * } \cO _{ \X } ( p ),\\
 ( p - 1, p - 1 ), 0 & \mbox{if} \  q - p = 0 \mbox{ and } \cU \not\simeq v ^{ * } \cO _{ \X } ( p ),\\
 ( p - 1, p ), 1 & \mbox{if} \  q - p = 1,\\
 ( p, q - 2 ), q - p - 2 & \mbox{if} \  q - p \ge 2.\\
 \end{cases}
\end{align}
\end{lemma}

\begin{proof}
Replacing
$
 \cU
$
with
$
 \cU \otimes v ^{ * } \cO _{ \X } ( - p )
$,
we may and will assume
$
 p = 0
$
throughout the proof.
Note first the constraint
\begin{align}\label{eq:q_is_a+b+2}
 q = \chi \lb \cU \rb = \chi \lb v _{ * } \cU \rb = a + b + 2,
\end{align}
where the first equality follows from the Riemann-Roch formula \eqref{eq:Riemann-Roch_for_embedded_curves}.

Take the normalization morphism
$
 \nu \colon W ^{ \nu } \simeq \X \coprod \X \to W
$,
so that
$
 \nu ^{ * } \cU \simeq \cO _{ \X } \coprod \cO _{ \X } ( q )
$.
Note that there exists the exact sequence
\begin{align}\label{eq:locally_free_U_adjunction_unit_map}
 0 \to \cU \to \nu _{ * } \nu ^{ * } \cU \to \nu _{ * } \cO _{ W ^{ \nu } } / \cO _{ W } \simeq \cO _{ W } / \cI
 \eqqcolon C \to 0,
\end{align}
where
$
 \cI \coloneqq \cHom _{ W } \lb \nu _{ * } \cO _{ W ^{ \nu } }, \cO _{ W } \rb \subset \cO _{ W }
$
is the conductor ideal. Depending on whether $W = I _{ 2 }$ or $III$, we have
$
 \cO _{ W } / \cI \simeq \bfk \times \bfk
$
or
$
 \bfk [ t ] / ( t ^{ 2 } )
$,
respectively.

Applying
$
 v _{ * }
$
to
\eqref{eq:locally_free_U_adjunction_unit_map}, we obtain the exact sequence
\begin{align}\label{eq:locally_free_U_pushforward_of_adjunction_unit_map}
 0
 \to
 v _{ * } \cU \simeq \cO _{ \X } ( a ) \oplus \cO _{ \X } ( b )
 \xto[]{\iota}
 \cO _{ \X } \oplus \cO _{ \X } ( q )
 \to
 v _{ * } C \simeq C \to 0.
\end{align}

Now suppose that
$
 q = 0
$, so that
$
 \cU \in \Pic ^{ 0 } ( W )
$.
Since the map
$
 \iota
$
of \eqref{eq:locally_free_U_pushforward_of_adjunction_unit_map} has trivial kernel, it follows that
$
 b \le 0
$.
Combining this with the assumption
$
 a \le b
$
and
\eqref{eq:q_is_a+b+2}, we immediately see
$
 ( a, b )
$
is either
$
 ( - 2, 0 )
$
or
$
 ( - 1, - 1 )
$.
It is then easy to observe that
$
 \cU \simeq \cO _{ W }
 \iff
 h ^{ 0 } ( W, \cU ) = 1
 \iff
 h ^{ 0 } ( \X, v _{ * } \cU ) = 1
 \iff
 \lb a, b \rb = ( - 2, 0 )
$.

Next assume that
$
 q > 0
$.
Apply
$
 \bR \Gamma \lb \X, - \rb
$
to \eqref{eq:locally_free_U_pushforward_of_adjunction_unit_map} and take the associated long exact sequence, to obtain:
\begin{align}
\begin{aligned}
 0
 \to
 H ^{ 0 } \lb \cO _{ \X } ( a ) \rb \oplus H ^{ 0 } \lb \cO _{ \X } ( b ) \rb
 \to
 H ^{ 0 } \lb \cO _{ \X } \rb \oplus H ^{ 0 } \lb \cO _{ \X } ( q ) \rb
 \xto[]{r}
 H ^{ 0 } \lb C \rb = C\\
 \to
 H ^{ 1 } \lb \cO _{ \X } ( a ) \rb \oplus H ^{ 1 } \lb \cO _{ \X } ( b ) \rb
 \to
 0
\end{aligned}
\end{align}

Note that the map
$
 r
$
is identified with the restriction map
\begin{align}
 H ^{ 0 } \lb W, \nu _{ * } \nu ^{ * } \cU \rb \to C,
\end{align}
which is surjective when
$
 q > 0
$.
This then implies
$
 H ^{ 1 } \lb \cO _{ \X } ( a ) \rb \oplus H ^{ 1 } \lb \cO _{ \X } ( b ) \rb = 0
$,
so that
$
 - 1 \le a ( \le b )
$.
When
$
 q = 1
$,
combining this again with \eqref{eq:q_is_a+b+2}, we immediately see that
$
 ( a, b ) = ( - 1, 0 )
$.

Consider the remaining case
$
 q \ge 2
$.
We see that
$
 a \le 0
$,
since otherwise the composition of the map
$
 \iota
$
of
\eqref{eq:locally_free_U_pushforward_of_adjunction_unit_map} with the projection to the component
$
 \cO _{ \X }
$
is trivial, contradicting that
$
 \coker \iota
$
is a torsion sheaf. Finally, suppose for a contradiction that
$
 a = - 1
$.
Then by applying
$
 - \otimes \cO _{ \X } ( - b )
$
to
\eqref{eq:locally_free_U_pushforward_of_adjunction_unit_map} and taking the associated long exact sequence, we obtain the following exact sequence.
\begin{align}
 0
 \to
 H ^{ 0 } \lb \cO _{ \X } ( - b - 1 ) \oplus \cO _{ \X } \rb
 \to
 H ^{ 0 } \lb \cO _{ \X } ( - b ) \oplus \cO _{ \X } ( 1 ) \rb
 \xto[]{ r ' }
 C.
\end{align}
As before, one can check that the map
$
 r '
$
is an isomorphism by identifying it with the restriction map
\begin{align}
 H ^{ 0 } \lb W, \nu _{ * } \nu ^{ * } \cU \otimes v ^{ * } \cO ( - b ) \rb
 \to
 C.
\end{align}
Thus we obtain a contradiction, concluding the proof.
\end{proof}

We next consider the case when
$
 \cU
$
is not an invertible sheaf on $W$. Since
$
 \cU
$
is torsion free of rank $1$, it is invertible on the smooth locus of $W$.
Moreover, as we observed in the beginning of \pref{sc:sbc}, the sheaf $\cU$ is a maximally Cohen-Macaulay (MCM) module over $W$.
Since the singularity of $W$ is either of type $A _{ 1 }$, $A _{ 2 }$, or $A _{ 3 }$, we can use the (finite!) classification of the indecomposable MCM modules on those singularities to classify the local structure of $\cU$ around the singularity of $W$, which is described for example in \cite[(5.12)]{MR1079937} (for
$
 A _{ \text{odd} }
$)
and
\cite[(9.9)]{MR1079937} (for
$
 A _{ \text{even} }
$).
Taking into account that the support of $\cU$ coincides with $W$, the local structures of $\cU$ around the singularities of $W$ are classified as follows.

\begin{align}
 \begin{array}{c|c|c}
\mbox{$W$} & \mbox{Singularity of $W$} & \mbox{
\begin{minipage}{6cm}
Stalk of $\cU$ at the singularity up to isomorphism
\end{minipage}
}\\
\hline
\hline
I _{ 1 } \  \mbox{or} \ I _{ 2 }
 & R = \bfk \ld x, y \rd _{ ( x, y ) } / \lb y ^{ 2 } - x ^{ 2 } \rb \quad \lb A _{ 1 } \rb & R \ \mbox{or} \  ( x, y )\\
\hline
II & R = \bfk \ld x, y \rd _{ ( x, y ) } / \lb y ^{ 2 } - x ^{ 3 } \rb \quad \lb A _{ 2 } \rb & R \ \mbox{or} \  ( x, y )\\
\hline
III \ ( \Atilde _{ 1 } ) & R = \bfk \ld x, y \rd _{ ( x, y ) } / \lb y ^{ 2 } - x ^{ 4 } \rb \quad \lb A _{ 3 } \rb & R \ \mbox{or} \  ( x, y ) \ \mbox{or} \ ( x ^{ 2 }, y )\\
\end{array}
\end{align}

\begin{remark}\label{rm:MCM_are_conductor_ideal}
The MCM modules on
$
 A _{ n }
$
singularities have an interpretation as the conductor ideals of the partial resolutions of the singularity.
We illustrate this in the case when
$
 n = 3
$.
Consider the following sequence of commutative $\bfk$-algebras.
\begin{align}
 S _{ 2 } = \bfk \ld u \rd \times \bfk \ld v \rd,\\
 S _{ 1 } = \bfk \ld u + v, u - v \rd \subset S _{ 2 },\\
 S _{ 0 } = \bfk \ld u + v, u ^{ 2 } - v ^{ 2 } \rd \subset S _{ 1 }.
\end{align}
By an abuse of notation we will denote the localization of these rings by the multiplicative set
$
 S _{ 0 } \setminus \lb u + v, u ^{ 2 } - v ^{ 2 } \rb S _{ 0 }
$
by the same symbols. If we let
$
 \frakm _{ S }
$
denote the maximal ideal of a local ring
$
 S
$,
then we have
$
 S _{ 1 } \simeq \End _{ S _{ 0 } } \frakm _{ S _{ 0 } }
$
and
$
 S _{ 2 } \simeq \End _{ S _{ 1 } } \frakm _{ S _{ 1 } }
$.
Then under the isomorphism
$
 R _{ III } \simto S _{ 0 }; \quad x, y \mapsto u + v, u ^{ 2 } - v ^{ 2 }
$,
the MCM modules
$
 R _{ III }, ( x, y ), ( x ^{ 2 }, y )
$
of
$
 R _{ III }
$
are identified with the ideals
\begin{align}
  \Hom _{ S _{ 0 } } \lb S _{ 0 }, S _{ 0 } \rb ( = S _{ 0 } ),\\
 \Hom _{ S _{ 0 } } \lb S _{ 1 }, S _{ 0 } \rb,\\
 \Hom _{ S _{ 0 } } \lb S _{ 2 }, S _{ 0 } \rb,
\end{align}
respectively.
\end{remark}

Suppose that
$
 I \subset W
$
is an ideal sheaf which is isomorphic to
$
 \cU
$
around the singularity and cosupported in
$
 \Sing W
$, which always exists by \pref{rm:MCM_are_conductor_ideal}.
In order to resolve the non-invertibility of $\cU$, consider the blowup
\begin{align}\label{eq:blowup}
 f \colon \Wtilde \coloneqq \Bl _{ I } W = \Proj _{ W } \bigoplus _{ d \ge 0 } I ^{ d } \to W.
\end{align}
It follows from the definition that
$
 f
$
is projective and an isomorphism over the smooth locus of $W$, which hence is surjective.
Also we have an invertible sheaf
$
 \scUtilde
$
on
$
 \Wtilde
$
such that
$
 f _{ * } \scUtilde \simeq \cU
$.
An explicit computation will tell us that
$
 f
$
locally normalizes $W$ at the points where
$
 \cU
$
is not invertible, except when
$
 W
$
is the curve
$
 III
$
and the ideal
$
 I
$
is locally isomorphic to
$
 ( x, y )
$
around the singularity, in which case
$
 f
$
is locally isomorphic to
$
 \Spec S _{ 1 } \to \Spec S _{ 0 }
$
and
$
 \Wtilde
$
is isomorphic to the union of two copies of $\X$ meeting transversally in a point.

\begin{lemma}\label{lm:irreducible_W_non_invertible_U}
Suppose that
$
 W
$
is irreducible (i.e.,
$
 I _{ 1 }
$
or
$
 II
$)
and
$
 \cU
$
is not invertible, so that there exists an invertible sheaf
$
 \scUtilde = \cO \lb i \rb
$
on the normalization
$
 f \colon W ^{ \nu } \simeq \X \to W
$
such that
$
 f _{ * } \cO \lb i \rb = \cU
$.
Then
\begin{align}
 ( a, b ), b - a
 =
 \begin{cases}
 \lb \frac{ i }{ 2 } - 1, \frac{ i }{ 2 } \rb, 1 & \mbox{if} \ i \equiv 0 \mod 2 \\
 \lb \frac{ i - 1 }{ 2 }, \frac{ i - 1 }{ 2 } \rb, 0 & \mbox{if} \ i \equiv 1 \mod 2.
 \end{cases}
\end{align}
\end{lemma}

\begin{proof}
We can again assume that
$
 i = 0
$
or
$
 1
$.
As before, we have
$
 \chi \lb \cU \rb
 =
 \chi \lb \cO \lb i \rb \rb
 =
 i + 1
 =
 a + b + 2
$.
We also have
$
 0
 =
 h ^{ 0 } \lb \cO _{ W ^{ \nu } } ( i - 2 ) \rb
 =
 h ^{ 0 } \lb \cO _{ \X } ( a - 1 ) \oplus \cO _{ \X } ( b - 1 ) \rb
$,
so that
$
 a \le 0
$
and
$
 b \le 0
$.
Therefore
$
 \lb a, b \rb
 =
 \lb - 1, 0 \rb
$
(if
$
 i = 0
$)
and
$
 \lb 0, 0 \rb
$
(if
$
 i = 1
$).
\end{proof}

When
$
 W
$
is not irreducible (i.e., either
$
 I _{ 2 }
$
or
$
 III
$),
it is easy to classify
$
 \cU
$
which is \emph{not} invertible, since in this case
$
 \Wtilde
$
of
\eqref{eq:blowup} has trivial
$
 H ^{ 1 } \lb \cO \rb
$,
so that
$
 \cU = f _{ * } \scUtilde
$
is uniquely determined by the (multi-)degree of
$
 \scUtilde
$.
We give a more explicit description as follows.

\begin{lemma}\label{lm:non_irreducible_W_non_invertible_U}
The classifications of non-invertible
$
 \cU
$
on a non-irreducible (and reduced)
$
 W
$
is given by the following table.
\begin{align}
\begin{array}{c|c|c|c|c}
 W & \mbox{Type of $\cU$} & \Wtilde & \scUtilde \quad ( p \le q ) & ( a, b ) \ \mbox{and} \  b - a\\
 \hline
 \hline
 I _{ 2 } & \mbox{Not invertible at a point} & \mbox{nodal conic} & \cO ( p, q ) &
 \begin{cases}
 ( p - 1, p ), 1 & ( \mbox{if $p = q$} )\\
 ( p, q - 1 ), q - p - 1 & ( \mbox{if $p < q$} )
 \end{cases}\\
 \hline
 III & \mbox{$\simeq ( x, y )$ around $\Sing W$} & \mbox{nodal conic}  & \cO ( p, q ) &
 \begin{cases}
 ( p - 1, p ), 1 & ( \mbox{if $p = q$} )\\
 ( p, q - 1 ), q - p - 1 & ( \mbox{if $p < q$} )
 \end{cases}\\
 \hline
 I _{ 2 } & \mbox{Not invertible at two points} & \X \coprod \X  & \cO ( p ) \coprod \cO ( q ) &
 ( p, q ), q - p\\
 \hline
 III & \mbox{$\simeq ( x ^{ 2 }, y )$ around $\Sing W$} & \X \coprod \X  & \cO ( p ) \coprod \cO ( q ) &
 ( p, q ), q - p\\
\end{array}
\end{align}
\end{lemma}

\begin{remark}
The last two cases are the so-called "decomposable case".
\end{remark}

From the classification of locally free sheaf bimodules which we have just completed, we obtain the following corollary.

\begin{corollary}
If a locally free sheaf bimodule
$
 \cE
$
satisfies
$
 b - a \ge 3
$,
then
$
 W
$
is not irreducible.
\end{corollary}


The following corollary will be used to check the strongness of certain full exceptional collections of the derived category of
$
 \qgr \bS ( \cE )
$.
\begin{corollary}\label{cr:a'_b'}
Set
$
 \cO ( - 1 ) \otimes \cE
 \simeq
 \cO _{ \X } \lb a' \rb \oplus \cO _{ \X } \lb b' \rb
$
with
$
 a ' \le b '
$,
so that
$
 a ' + b '
 =
 \chi \lb \cU \rb - 4
$.

\begin{enumerate}
\item
If
$
 W
$
is not reduced and
\begin{enumerate}[(i)]
\item
$
 \cU
$
is invertible, then
\begin{align}
 ( a ', b ' ), b' - a'
 =
 \begin{cases}
 \lb \frac{\deg \cU }{ 2 } - 3, \frac{\deg \cU }{ 2 } - 1 \rb, 2
 & \cU \simeq v ^{ * } \cO _{ \X } \lb k \rb \ \lb \exists k \in \bZ \rb, \\
 \lb \frac{\deg \cU }{ 2 } - 2, \frac{\deg \cU }{ 2 } - 2 \rb, 0
 & \cU \not \simeq v ^{ * } \cO _{ \X } \lb k \rb \ \lb \forall k \in \bZ \rb. \\
 \end{cases} 
\end{align}

\item
$
 \cU
$
is not invertible, so that
$
 \deg D > 0
$, then
\begin{align}
 ( a ', b ' ), b ' - a '
 =
\begin{cases}
 \lb \frac{ \chi (\cU) - 5 }
 { 2 }, \frac{ \chi (\cU) - 3 }
 { 2 } \rb,
 1 & \deg D = 1\\
 \lb \frac{ \chi (\cU) - \deg D }
 { 2 } 
 - 1, \frac{ \chi (\cU) +\ \deg D }
 { 2 } - 3 \rb,
 \deg D - 2 & \deg D \ge 2.
\end{cases}
\end{align}
\end{enumerate}

\item
If
$
 W
$
is integral
(i.e.,
$
 I _{ 0 }, I _{ 1 }
$
or
$
 II
$)
and
$
 \cU
$
is invertible, then
\begin{align}
 ( a ', b ' ), b' - a'
 =
 \begin{cases}
 \lb \frac{\deg \cU}{2} - 3, \frac{\deg \cU}{2} - 1 \rb, 2 & u ^{ * } \cO _{ \X } \lb - 1 \rb \otimes \cU \simeq v ^{ * } \cO _{ \X } \lb k \rb \
 ( \exists k \in \bZ )\\
 \lb \frac{\deg \cU}{2} - 2, \frac{\deg \cU}{2} - 2 \rb, 0 & u ^{ * } \cO _{ \X } \lb - 1 \rb \otimes \cU \not\simeq v ^{ * } \cO _{ \X } \lb k \rb \
 ( \exists k \in \bZ )\\
 \lb \frac{\deg \cU - 1}{2} - 2, \frac{\deg \cU - 1}{2} - 1 \rb, 1 & \deg \cU \equiv 1 \mod 2
 \end{cases}
\end{align}

\item
If
$
 W
$
is integral and
$
 \cU
$
is not invertible, then
\begin{align}
 \lb a ', b ' \rb, b' - a'
 =
 \begin{cases}
 \lb \frac{ i }{ 2 } - 2, \frac{ i }{ 2 } - 1 \rb, 1 & i \equiv 0 \mod 2,\\
 \lb \frac{ i - 1 }{ 2 } - 1, \frac{ i - 1 }{ 2 } - 1 \rb, 0 & i \equiv 1 \mod 2.
 \end{cases}
\end{align}

\item
If
$
 W
$
is not irreducible (i.e.,
$
 I _{ 2 }
$
or
$
 III
$)
and
$
 \cU
$
is invertible, then
\begin{align}
 \lb a ', b ' \rb, b' - a'
 =
 \begin{cases}
 \lb p - 3, p - 1 \rb, 2
 & q - p = 0 \mbox{ and } \cU \otimes u ^{ * } \cO _{ \X } ( - 1 ) \simeq v ^{ * } \cO _{ \X } ( p - 1 ),\\
 \lb p - 2, p - 2 \rb, 0
 & q - p = 0 \mbox{ and } \cU \otimes u ^{ * } \cO _{ \X } ( - 1 ) \not\simeq v ^{ * } \cO _{ \X } ( p - 1 ),\\
 \lb p - 2, p - 1 \rb, 1 & q - p = 1,\\
 \lb p - 1, q - 3 \rb, q - p - 2 & q - p \ge 2.\\
 \end{cases}
\end{align}

\item
If
$
 W
$
is not irreducible (i.e.,
$
 I _{ 2 }
$
or
$
 III
$)
and
$
 \cU
$
is not invertible, then
\begin{enumerate}[(i)]
\item
When
$
 \Wtilde \simeq X
$, then
\begin{align}
 \lb a ', b ' \rb, b' - a'
 =
 \begin{cases}
 \lb p - 2, p - 1 \rb, 1 & \lb \mbox{if $ p = q $} \rb\\
 \lb p - 1, q - 2 \rb, q - p - 1 & \lb \mbox{if $ p < q $} \rb
 \end{cases}
\end{align}

\item
When
$
 \Wtilde \simeq W ^{ \nu } \simeq \X \coprod \X
$, then
\begin{align}
 \lb a ', b ' \rb, b' - a'
 =
 \lb p - 1, q - 1 \rb, q - p
\end{align}
\end{enumerate}
\end{enumerate}
\end{corollary}

\begin{proof}
Set
$
 \cM \coloneqq u ^{ * } \cO \lb - 1 \rb \otimes v ^{ * } \cO ( 1 ) \in \Pic ^{ 0 } W
$.
Then one observes
\begin{align}
 \cO ( - 1 ) \otimes \cE
 =
 v _{ * } \lb u ^{ * } \cO ( - 1 ) \otimes \cU \rb
 =
 v _{ * } \lb v ^{ * } \cO ( - 1 ) \otimes \cU \otimes \cM \rb
 \simeq
 v _{ * } \lb \cU \otimes \cM \rb \otimes \cO ( - 1 ),
\end{align}
where the last isomorphism is the projection formula.
Then one immediately obtains the conclusions just by applying our results above to
$
 \cU \otimes \cM
$.
\end{proof}

\section{Moduli stack of sheaf bimodules}
\label{sc:stability_and_deformation_theory_of_sheaf_bimodules}
\subsection{Gieseker stability}

Let
$
 \cE = \iota _{ * } \cU
$
be a locally free sheaf bimodule on
$
 \X
$.
In this section we completely determine the Gieseker stability of sheaf bimodules with respect to the polarization
$
 - K _{ \X \times \X } = \cO _{ \X \times \X } ( 2, 2 )
$.
We recall the definition of Gieseker stability from \cite{MR2665168}.
\begin{definition}
Let
$
 ( X, H )
$
be a polarized projective scheme over a field $\bfk$.
For a coherent sheaf
$
 E \in \coh X
$
on $X$, 
the Hilbert polynomial
$
 P _{ E } ( t )
$
of $E$ with respect to $H$ is the polynomial in
$
 \bQ [ t ]
$
which satisfies
$
 P _{ E } ( m ) = \chi \lb E \otimes \cO _{ X } ( m H ) \rb
$
for all
$
 m \in \bZ
$.
When
$
 \dim \Supp E = d
$,
there are rational numbers
$
 a _{ i } \in \bQ
$
for
$
 i = 0, 1, \dots, d
$
such that
$
 a _{ d } > 0
$
and
$
 P _{ E } ( t ) = \sum _{ i = 0 } ^{ d } a _{ i } \frac{ t ^{ i } }{ i ! }
$.
The \emph{reduced Hilbert polynomial} of $E$ is defined to be
$
 p _{ E } ( t )
 =
 \frac{1}{a _{ d }} P ( t )
$.

A pure sheaf
$
 E
$
is said to be Gieseker (semi-)stable with respect to the polarization
$
 H
$
if the inequality
\begin{align}\label{eq:inequality_reduced_Hilbert_polynomials}
 p _{ F } ( t ) < ( \le ) p _{ E } ( t )
\end{align}
holds for any subsheaf
$
 0 \neq F \subsetneq E
$, where \eqref{eq:inequality_reduced_Hilbert_polynomials} should be interpreted that the leading coefficient of the polynomial
$
 p _{ E } ( t ) - p _{ F } ( t )
$
is strictly positive (non-negative, respectively).
\end{definition}

Note that if
$
 \iota \colon Y \hookrightarrow X
$
is a closed immersion, then
$
 E \in \coh Y
$
is
$
 \iota ^{ * } H
$
(semi-)stable if and only if so is
$
 \iota _{ * } E
$
with respect to
$
 H
$.

The following lemma immediately follows from the definition.
\begin{lemma}
Let
$
 \cE
$
be an indecomposable locally free sheaf bimodule over
$
 \X
$
of type \eqref{it:type_1_1}.
Then
$
 \cU
$
(hence
$
 \cE
$)
is unstable if
$
 a < b
$, and is semi-stable but not stable if
$
 a = b
$.
\end{lemma}

In the rest of this subsection, we consider the case \eqref{it:type_2_2}.

\begin{proposition}\label{pr:stability_of_cU_nonreduced_W}
Let
$
 \cE
$
be a locally free sheaf bimodule over
$
 \X
$
whose scheme theoretic support
$
 W
$
is not reduced. With the notation of \pref{th:sheaf_bimodule_non-reduced}, the stability of
$
 \cU
$
with respect to the polarization
$
 \cO _{ W } ( - K _{ \X \times \X } )
$
is classified as follows, depending on
$
 \deg D
$.

\begin{align}
 \begin{array}{c|c}
 \deg D & \mbox{Stability of $\cU$}\\
 \hline
 \ge 3 & \mbox{unstable}\\
 2 & \mbox{semi-stable but not stable}\\
 \mbox{$0$ or $1$} & \mbox{stable}
 \end{array}
\end{align}
\end{proposition}

\begin{proof}
Since
$
 \cL
$
is an invertible sheaf on $W$, the stability of
$
 \cU
$
coincides with that of
$
 \cU \otimes \cL ^{ - 1 }
$. Hence we may and will assume
$
 \cL = \cO _{ W }
$,
without loss of generality.
Then we have the short exact sequence \eqref{eq:decomposition_of_U_L_-1}, with the middle term
$
 = \cU
$.

By the direct computation, we obtain
\begin{align}
 p _{ \cU } ( m )
 =
 m - \frac{1}{8} \deg D
\end{align}
and
\begin{align}
 p _{ i _{ * } \cO _{ \Wred } ( - 2 ) } ( m )
 =
 m - \frac{1}{4}
\end{align}
Therefore the subsheaf
$
 i _{ * } \cO _{ \Wred } ( - 2 ) \subset \cU
$
ensures that $\cU$ is not stable if
$
 \deg D \ge 2,
$
and is unstable if
$
 \deg D \ge 3
$.
From now on, let us assume
$
 \deg D \le 2
$.
Take an arbitrary subsheaf
$
 0 \neq F \subsetneq \cU
$, and consider the exact sequence
\begin{align}\label{eq:standard_sequence_to_check_stability_of_cU}
 0
 \to
 F
 \xto[]{j}
 \cU
 \xto[]{p}
 \cU / F
 \to
 0.
\end{align}
If
$
 \dim \Supp \cU / F \le 0
$,
then
$
 P _{ \cU } ( t ) = P _{ F } ( t ) + C
$,
where
$
 C = \dim _{ \bfk } \cU / F
$
is a positive constant. In this case, clearly
$
 p _{ F } < p _{ \cU }
$.
Suppose that
$
 \dim \Supp \cU / F = 1
$.

Consider the case when
$
 F \cap i _{ * } \cO _{ \Wred } ( - 2 ) \neq 0
$.
Then the assumption implies
$
 \dim \Supp p ( F ) \le 0
$,
so that
$
 p ( F ) = 0
$
and hence
$
 F \subset i _{ * } \cO _{ \Wred } ( - 2 )
$.
Therefore in this case
$
 p _{ F } \le p _{ \cU }
$
when
$
 \deg D = 2
$,
and
$
 p _{ F } < p _{ \cU }
$
when
$
 \deg D \le 1
$.

Consider then the other case when
$
 F \cap i _{ * } \cO _{ \Wred } ( - 2 ) = 0
$,
so that
$
 p | _{ F } \colon F \simto p ( F ) \subset i _{ * } \cO _{ \Wred } ( - D )
$.
At this point we can already conclude that
$
 \cU
$
is semi-stable when
$
 \deg D = 2
$.

Note that
$
 p ( F ) \subsetneq i _{ * } \cO _{ \Wred } ( - D )
$,
since otherwise $p$ splits so as to contradict the assumption that
$
 W
$
is the scheme theoretic support of $\cU$. Finally, by the standard long exact sequence argument we can check that
$
 \Hom _{ W } \lb i _{ * } \cO _{ \Wred } ( - 1 ), \cU \rb = 0
$,
to conclude that
$
 \cU
$
is stable when
$
 \deg D \le 1
$.
\end{proof}

\begin{proposition}
Let
$
 \cE
$
be an indecomposable locally free sheaf bimodule over
$
 \X
$
whose scheme theoretic support
$
 W
$
is reduced and irreducible (i.e., either
$
 I _{ 0 }, I _{ 1 }
$,
or
$
 II
$).
Then
$
 \cE
$
is stable with respect to any polarization.
\end{proposition}

\begin{proof}
Let
$
 H
$
be a polarization of
$
 W
$.
We show the stability of
$
 \cU
$
with respect to $H$.

Take any subsheaf
$
 0 \neq F \subsetneq \cU
$,
and consider the short exact sequence
\eqref{eq:standard_sequence_to_check_stability_of_cU}.
The assumption implies that
$
 W
$
is an integral scheme, and since
$
 \cU
$
has rank $1$, it follows that the stalk of the map
$
 F \xto[]{j} \cU
$
at the generic point of $W$ is an isomorphism, so that
$
 \dim \Supp \cU / F \le 0
$.
As is already discussed in the proof of the previous proposition,
from this we immediately obtain the conclusion.
\end{proof}

\begin{proposition}
Let
$
 \cE
$
be a locally free sheaf bimodule over
$
 \X
$
whose scheme theoretic support
$
 W
$
is not irreducible (i.e., either
$
 I _{ 2 }
$
or
$
 III
$), and
$
 \cU
$
is an invertible sheaf of bidegree
$
 \lb p, q \rb
$
on
$
 W
$
with
$
 p \le q
$.
Then the stability of
$
 \cU
$
with respect to the polarization
$
 \cO _{ W } ( - K _{ \X \times \X } )
$
is classified as follows, depending on
$
 q - p
$.

\begin{align}
 \begin{array}{c|c}
 q - p & \mbox{Stability of $\cU$}\\
 \hline
 \ge 3 & \mbox{unstable}\\
 2 & \mbox{semi-stable but not stable}\\
 \mbox{$0$ or $1$} & \mbox{stable}
 \end{array}
\end{align}
\end{proposition}

\begin{proof}
Let
$
 i _{ p }, i _{ q }
$
be the closed immersions into
$
 W
$
of its irreducible components on which the restriction of
$
 \cU
$
is isomorphic to
$
 \cO ( p )
$
and
$
 \cO ( q )
$, respectively. On $W$, there exists a short exact sequence
\begin{align}\label{eq:decomposition_I2_or_III_invertible}
 0
 \to
 i _{ q * } \cO ( q - 2 )
 \to
 \cU
 \to
 i _{ p * } \cO ( p )
 \to
 0.
\end{align}
This immediately implies that
$
 \cU
$
is unstable when
$
 q - p \ge 3
$
and is not stable when
$
 q - p = 2
$.

By the similar arguments as in the proof of \pref{pr:stability_of_cU_nonreduced_W}, the proof of the assertion amounts to showing
$
 \Hom _{ W } \lb i _{ p * } \cO ( p - 1 ), \cU \rb = 0
$. The rest of the proof will be devoted to this step.

Applying
$
 \Hom _{ W } \lb i _{ p * } \cO ( p - 1 ), - \rb
$
to the sequence \eqref{eq:decomposition_I2_or_III_invertible}, as part of the long exact sequence we obtain
\begin{align}
\begin{aligned}
 0
 \to
 \Hom _{ W } \lb i _{ p * } \cO ( p - 1 ), \cU \rb
 \to
 \Hom _{ W } \lb i _{ p * } \cO ( p - 1 ), i _{ p * } \cO ( p ) \rb\\
 \to
 \Ext ^{ 1 } _{ W } \lb i _{ p * } \cO ( p - 1 ), i _{ q * } \cO ( q - 2 ) \rb
 \to
 \Ext ^{ 1 } _{ W } \lb i _{ p * } \cO ( p - 1 ), \cU \rb.
\end{aligned}
\end{align}
Note first that
$
 \Hom _{ W } \lb i _{ p * } \cO ( p - 1 ), i _{ p * } \cO ( p ) \rb \simto \bfk ^{ 2 }
$.
To compute the next term, note that
\begin{align}
 \Ext ^{ 1 } _{ W } \lb i _{ p * } \cO ( p - 1 ), i _{ q * } \cO ( q - 2 ) \rb
 \simeq
 H ^{ 0 } \lb W, \cExt _{ W } ^{ 1 } \lb i _{ p * } \cO ( p - 1 ), i _{ q * } \cO ( q - 2 ) \rb \rb,
\end{align}
since
$
 \cExt _{ W } ^{ 1 } \lb i _{ p * } \cO ( p - 1 ), i _{ q * } \cO ( q - 2 ) \rb
$
is supported in the intersection of the two irreducible components and hence the local-to-global spectral sequence degenerates at sheet $2$.
A local computation ensures that the RHS is isomorphic to
$
 \bfk \times \bfk
$
or
$
 \bfk [ t ] / ( t ^{ 2 } )
$,
depending on whether $W$ is $I _{ 2 }$ or $III$. Hence in any case
$
 \dim _{ \bfk } \Ext ^{ 1 } _{ W } \lb i _{ p * } \cO ( p - 1 ), i _{ q * } \cO ( q - 2 ) \rb 
 =\dim _{\bfk}\bfk ^{ 2 } = 2
$.

Finally, by the similar reasoning one obtains
\begin{align}
 \Ext ^{ 1 } _{ W } \lb i _{ p * } \cO ( p - 1 ), \cU \rb
 \simeq
  H ^{ 0 } \lb W, \cExt _{ W } ^{ 1 } \lb i _{ p * } \cO ( p - 1 ), \cU \rb \rb.
\end{align}
However, from local computations we always obtain
$
 \cExt _{ W } ^{ 1 } \lb i _{ p * } \cO ( p - 1 ), \cU \rb = 0
$.
Thus we conclude the proof.
\end{proof}

\begin{proposition}
Let
$
 \cE
$
be a locally free sheaf bimodule over
$
 \X
$
whose scheme theoretic support
$
 W
$
is not irreducible (i.e., either
$
 I _{ 2 }
$
or
$
 III
$), and
$
 \cU
$
is a non-invertible sheaf of rank 1 of bidegree
$
 \lb p, q \rb
$
on
$
 W
$
with
$
 p \le q
$.
Then the stability of
$
 \cU
$
with respect to the polarization
$
 \cO _{ W } ( - K _{ \X \times \X } )
$
is classified as follows, depending on
$
 q - p
$.

\begin{align}
 \begin{array}{c|c|c}
 \Wtilde & q - p & \mbox{Stability of $\cU$}\\
 \hline
 \mbox{nodal conic} & \ge 2 & \mbox{unstable}\\
 & 1 & \mbox{semi-stable but not stable}\\
 & 0 & \mbox{stable}\\
 \hline
 \X \coprod \X & \ge 1 & \mbox{unstable}\\
 & 0 & \mbox{semi-stable but not stable} 
 \end{array}
\end{align}
\end{proposition}

\begin{proof}
The assertion for the case
$
 \Wtilde \simeq \mbox{nodal conic}
$
(two copies of $\X$ glued transversally together at a point) can be proved by the same arguments as in the proof of the previous proposition. The assertion for the case
$
 \Wtilde \simeq \X \coprod \X
$
is rather obvious.
\end{proof}

\subsection{Deformation theory}
The deformation theory of the sheaf bimodule
$
 \cE
$
as a coherent sheaf on
$
 \X \times \X
$
is controlled by the differential graded Lie algebra
$
 \Hom ^{ \bullet } _{ \X \times \X } \lb \cE, \cE \rb
$.
As a corollary, the first order infinitesimal automorphisms and the first order infinitesimal deformations of $\cE$ are classified by
$
 \Ext ^{ 0 } _{ \X \times \X } \lb \cE, \cE \rb
$
and
$
 \Ext ^{ 1 } _{ \X \times \X } \lb \cE, \cE \rb
$,
respectively, and
$
 \Ext ^{ 2 } _{ \X \times \X } \lb \cE, \cE \rb
$
serves as an obstruction space.
Based on our stability analysis above, we can partially compute these spaces as follows.

\begin{proposition}\label{pr:ext_groups_of_sheaf_bimodules}
Let
$
 \cE
$
be a locally free sheaf bimodule on $\X$.
\begin{enumerate}
\item
It holds that
\begin{align}
 \chi \lb \cE, \cE \rb
 =
 \sum _{ i = 0 } ^{ 2 } ( - 1 ) ^{ i } \dim _{ \bfk } \Ext ^{ i } _{ \X \times \X } \lb \cE, \cE \rb
 =
 - 8.
\end{align}

\item
When
$
 \cE
$
is semi-stable with respect to the polarization
$
 \cO _{ \X \times \X } \lb - K _{ \X \times \X } \rb
$,
then
$
 \Ext ^{ 2 } _{ \X \times \X } \lb \cE, \cE \rb = 0
$.
In particular, the deformation of
$
 \cE
$
is unobstructed.

\item
When
$
 \cE
$
is stable with respect to the polarization
$
 \cO _{ \X \times \X } \lb - K _{ \X \times \X } \rb
$
or the sheaf
$
 \cU
$
on
$
 W
$
is invertible, then
\begin{align}
 \lb \dim _{ \bfk } \Ext ^{ i } _{ \X \times \X } \lb \cE, \cE \rb \rb _{ i = 0, 1, 2 }
 =
 \lb 1, 9, 0 \rb.
\end{align}
\end{enumerate}
\end{proposition}

\begin{proof}
\begin{enumerate}
\item
Since
$
 \dim \X \times \X = 2
$,
the extension groups between coherent sheaves are concentrated in degrees
$
 0, 1, 2
$.

It follows from the classification that any locally free sheaf bimodule on
$
 \X
$
is deformation equivalent to the standard one
$
 \cO _{ \Delta } ( a ' ) \oplus \cO _{ \Delta } ( b ' )
$
for some
$
 \lb a ' , b ' \rb
$.
Using the deformation invariance of
$
 \chi \lb \cE, \cE \rb
$
one can reduce the computation to that for the standard one, to obtain the result.

\item
By the Serre duality,
$
 \dim _{ \bfk } \Ext ^{ 2 } _{ \X \times \X } \lb \cE, \cE \rb
 =
 \dim _{ \bfk } \Ext ^{ 0 } _{ \X \times \X } \lb \cE, \cE \lb K _{ \X \times \X } \rb \rb
$.
Since
$
 p _{ \cE \lb K _{ \X \times \X } \rb } ( t )
 =
 p _{ \cE } ( t - 1 )
$, it always follows that
$
 p _{ \cE } > p _{ \cE \lb K _{ \X \times \X } \rb }
$.
Hence if
$
 \cE
$
is semi-stable, there is no non-trivial homomorphism from
$
 \cE
$
to
$
 \cE \lb K _{ \X \times \X } \rb
$
by \cite[Proposition 1.2.7]{MR2665168}.

\item
When
$
 \cE
$
is stable, then
$
 \Ext ^{ 0 } _{ \X \times \X } \lb \cE, \cE \rb = \bfk \id _{ \cE }
$
by
\cite[Corollary 1.2.8]{MR2665168}.
When
$
 \cU
$
is invertible we have
$
 \Hom _{ W } \lb \cU, \cU \rb
 \simeq
 H ^{ 0 } \lb W, \cO _{ W } \rb
$
and the latter is isomorphic to
$
 \bfk
$,
since
$
 W
$
is an effective Cartier divisor of $\X \times \X$.
Combining these with the first two items, we obtain the conclusion.
\end{enumerate}
\end{proof}

\subsection{Moduli stack}
 \label{sc:Msh}

Let
$
 \Msht'
$
be the stack of locally free sheaf bimodules on
$
 \X
$,
which is an open substack of the stack
of coherent sheaves on
$
 \X \times \X
$.
The automorphism group of every object of $\Msht'$
contains the multiplicative group $\bGm$,
and we write the rigidification as
\begin{align}
 \Msht \coloneqq \left. \Msht' \right/ \! \mathrm{B} \bGm,
\end{align}
which is written as $\lb \Msht' \rb^{\bGm}$ and
$\Msht' \!\! \fatslash \, \bGm$
in \cite{MR2007376} and \cite{MR2125542} respectively.

The stack $\Msht$ is decomposed into open and closed substacks by $\deg \cU$.
Note that the map $\cU \mapsto \cU \otimes v^* \cO_{\X}(k)$ induces
an isomorphism between components
whose degrees differ by an even integer, preserving the equivalence classes of the associated categories
$
 \Qgr
$.
We write the connected component parametrizing sheaf bimodules of degree 2 (resp. 3) as
$\Mshtz$ (resp. $\Mshto$).  

The discussion in the previous subsections shows that
$
 \cE
$
is simple and has $9$-dimensional unobstructed deformation space
if $W$ is non-singular.
Hence the dimension of the open substack of
$
 \Msht
$
consisting of $\cE$ with non-singular $W$ has dimension $9$.
Roughly speaking, 8 out of 9 comes from the linear system
\begin{align}
 |W| = | - K _{ \X \times \X } |
 =
 \bP _{ \bullet } H^0( \cO_{\X}(2) \boxtimes \cO_{\X}(2) )
 \simeq \bP^8,
\end{align}
and the remaining 1 comes from
$
 \Pic ^{ 0 } ( W )
$.
The group
$
 \Aut(\X) \times \Aut(\X) \simeq (\PGL_2)^2
$
acts naturally on the stack $\Msht$
by pull-back of coherent sheaves.
The dimension of the quotient stack
$\Msh \coloneqq [\Msht / (\PGL_2)^2]$
is the same as the expected dimension $3$ of the moduli space
of noncommutative Hirzebruch surfaces. We let
$
 \Mshz
$
and
$
 \Msho
$
denote the quotients of
$
 \Mshtz
$
and
$
 \Mshto
$
respectively.

\section{Quivers with relations from noncommutative Hirzebruch surfaces}
\label{sc:fda}

\subsection{Sheaf bimodule as the kernel of the dual gluing functor}
 \label{sc:SOD}

Let $\cD$ be a triangulated category.
A full triangulated subcategory $\cN \subset \cD$ is
\emph{right}
(resp. \emph{left})
\emph{admissible}
if the inclusion functor
$i \colon \cN \hookrightarrow \cD$
has a right (resp. left) adjoint functor
$i^!$ (resp. $i^*$).
It is \emph{admissible}
if it is both right and left admissible.
The
\emph{right}
(resp. \emph{left})
\emph{orthogonal}
$\cN^\bot$ (resp. $\lbot{\cN}$)
is the full subcategory of $\cD$
consisting of objects $X \in \cD$
satisfying $\Hom(N, X) = 0$
(resp. $\Hom(X, N) = 0$)
for any $N \in \cN$.
The subcategory $\cN$ is right (resp. left) admissible
if and only if for any $X \in \cD$,
there exists $N \in \cN$ and $M \in \cN^\bot$
(resp. $N' \in \cN$ and $M' \in \lbot{\cN}$)
forming a distinguished triangle
$
 N \to X \to M \to N[1]
$
(resp.
$
 M' \to X \to N' \to M'[1]
$).
Let $(\cN_1, \ldots, \cN_n)$
be a sequence of subcategories of $\cD$.
For $i = 1, \ldots, n$,
we let $\cD_i$ denote the smallest full triangulated subcategory of $\cD$
containing objects of $\cN_1, \ldots, \cN_i$.
The sequence $(\cN_1, \ldots, \cN_n)$ is
a \emph{semiorthogonal decomposition} of $\cD$
if 
\begin{itemize}
 \item
$\cD_{i-1}$ is right admissible in $\cD_i$ for $i=1, \ldots, n-1$,
 \item
the left orthogonal of $\cD_{i-1}$ in $\cD_i$ is equivalent to $\cN_i$,
and
 \item
$\cD_n$ is equivalent to $\cD$.
\end{itemize}
We write
$
 \cD = \la \cN_1, \ldots, \cN_n \ra
$
if $(\cN_1, \ldots, \cN_n)$ is a semiorthogonal decomposition of $\cD$.

Let
$\cD = \la \cN_1, \cN_2 \ra$
be a semiorthogonal decomposition,
so that the inclusion functor
$i_1 \colon \cN_1 \to \cD$ has a right adjoint $i_1^!$
and the inclusion functor
$i_2 \colon \cN_2 \to \cD$ has a left adjoint $i_2^*$.
The \emph{gluing bimodule} is given by
\begin{align}
 \hom_\cD(i_1 (-), i_2 (-))
  \simeq \hom_{\cN_2}(\phi(-), -)
  \simeq \hom_{\cN_1}(-, \phi^!(-))
  \colon \cD \to D^b(\bfk),
\end{align}
where
\begin{align}
 \phi = i_2^* \circ i_1 \colon \cN_1 \to \cN_2
\end{align}
is the \emph{gluing functor}
and
\begin{align}
 \phi^! = i_1^! \circ i_2 \colon \cN_2 \to \cN_1
\end{align}
is the \emph{dual gluing functor}.
The category $\cD$ can be recovered
from the categories $\cN_1$, $\cN_2$
and the gluing bimodule
by the upper-triangular matrix construction
\cite{1402.7364}.

The following is a generalization of \cite{MR1208153}.

\begin{theorem} \label{th:SOD}
Let $X$ be a smooth projective scheme and
$\cE$ be a locally free sheaf bimodule of rank 2 on $X$.
Then one has a semiorthogonal decomposition
\begin{align} \label{eq:SOD}
 D^b \qgr \bS(\cE) =
  \la f_1^* D^b \coh X, f_0^* D^b \coh X \ra,
\end{align}
and the dual gluing functor is given by
\begin{align}
 f_0^* D^b \coh X
  \xto{(f_0^*)^{-1}} D^b \coh X
  \xto{-\otimes_{\cO_X} \cE} D^b \coh X
  \xto {f_1^*} f_1^* D^b \coh X.
\end{align}
\end{theorem}

\begin{proof}
The subcategory $f_i^* D^b \coh X$ is right admissible
since $f_i^*$ has a right adjoint functor $\bR f_{i*}$.
The right admissibility also implies the left admissibility
since $D^b \qgr \bS(\cE)$ has a Serre functor. 
It follows from \eqref{eq:Mor07.4.4} that
the full subcategory $f_1^* D^b \coh X$
is contained in the right orthogonal of
$f_0^* D^b \coh X$.
For any $m \in \bZ$ and any locally-free $\cO_X$-module $\cF$,
one has an exact sequence
\begin{align}
 0
  \to \cF \otimes_{\cO_X} \cQ_m \otimes_{\cO_X} e_{m+2} \cA
  \to \cF \otimes_{\cO_X} \cE^{*m} \otimes_{\cO_X} e_{m+1} \cA
  \to \cF \otimes_{\cO_X} e_m \cA
  \to \cF
  \to 0
\end{align}
in $\Gr \bS(\cE)$
(see \cite[Theorem 1.4]{MR2115370}),
which induces an exact sequence
\begin{align} \label{eq:key1}
 0
  \to f_{m+2}^* (\cF \otimes_{\cO_X} \cQ_m)
  \to f_{m+1}^*(\cF \otimes_{\cO_X} \cE^{*m})
  \to f_m^* \cF
  \to 0
\end{align}
in $\qgr \cS(\cE)$.
It follows from \eqref{eq:key1} that
if a full triangulated subcategory $\cT$ of $D^b \qgr \bS(\cE)$ contains
$f_{m+2}^* D^b \coh X$ and $f_{m+1}^* D^b \coh X$ for some $m \in \bZ$,
then it also contains $f_m^* D^b \coh X$.
Similarly,
if a full triangulated subcategory $\cT$ of $D^b \qgr \bS(\cE)$ contains
$f_{m+1}^* D^b \coh X$ and $f_m^* D^b \coh X$,
then it also contains $f_{m+2}^* D^b \coh X$,
since the sheaf bimodule $\cQ_m$ is invertible.
It follows that
if a full triangulated subcategory $\cT$ of $D^b \qgr \bS(\cE)$ contains
$f_{m+2}^* D^b \coh X$ and $f_{m+1}^* D^b \coh X$ for some $m \in \bZ$,
then it also contains $f_n^* D^b \coh X$ for all $n \in \bZ$.
It follows from \cite[Proposition 2.19]{MR2115370}
that the subcategory of $D^b \qgr \bS(\cE)$
right orthogonal to $f_n^* D^b \coh X$ for all $n \in \bZ$ is zero,
and \eqref{eq:SOD} is proved.
The dual gluing functor is given on locally-free $\cO_X$-modules by
\begin{align}
 \bR {f_1}_* f_0^* \simeq (-) \otimes_{\cO_X} \cA_{01}
  \simeq (-) \otimes_{\cO_X} \cE
\end{align}
by \eqref{eq:Mor07.4.4}.
\end{proof}

\subsection{Full strong exceptional collections on noncommutative Hirzebruch surfaces}
\label{sc:fsec}
Let $\cE$ be a locally free sheaf bimodule of rank 2 on $\X$.
By \pref{th:SOD},
the derived category $D^b \qgr \bS(\cE)$
is obtained by gluing two copies of $D^b \coh \X$
by the dual gluing functor
$
 \phi^! \coloneqq (-) \otimes_{\cO_\X} \cE
  \colon D^b \coh \X \to D^b \coh \X.
$
If one sets
\begin{align} \label{eq:EC}
 (E_1, E_2, E_3, E_4) \coloneqq
  (f_1^* \cO_\X ( - m - 1 ), f_1^* \cO_\X( - m ),
   f_0^* \cO_\X ( - 1 ), f_0^* \cO_\X),
\end{align}
then this is a full exceptional collection for any
$
 m \in \bZ
$.

In this section, we specify the pairs
$
 \lb \cE, m \rb
$
for which
\eqref{eq:EC} is a strong exceptional collection.

If we set
\begin{align}\label{eq:E_and_F}
 (F_1, F_2, F_3, F_4) \coloneqq
 ( \cO_\X ( - m - 1 ),
 \cO_\X( - m ),
 \phi^! \cO_\X ( - 1 ),
 \phi^! \cO_\X ),
\end{align}
then one has
\begin{align}
 \RHom_\cD(E_i, E_j)
  &\simeq
\begin{cases}
 \RHom_\X(F_i, F_j) & 3 \neq i < j, \\
 \bfk ^{ 2 } [ 0 ] & ( i, j ) = ( 3, 4 ), \\
 \bfk [ 0 ] & i = j, \\
 0 & \text{otherwise}.
\end{cases}
\end{align}
Hence we obtain the following explicit characterization of strongness,
which immediately follows from
$
 \phi^! \cO_\X ( - 1 ) = \cO _{ \X } ( a ' ) \oplus \cO _{ \X } ( b ' )
$,
$
 \phi^! \cO_\X = \cO _{ \X } ( a ) \oplus \cO _{ \X } ( b ) )
$,
and the inequality
$
 a \ge a '
$,
which is a consequence of the explicit computations given in \pref{sc:sbc}.

\begin{theorem}\label{th:iff_condition_for_strongness}
\eqref{eq:EC} is strong if and only if
\begin{align}
 \Ext ^{ 1 } _{ \X } ( F_i, F_j ) = 0
\end{align}
for
$
 \lb i, j \rb =
 ( 1, 3 ),
 ( 1, 4 ),
 ( 2, 3 ),
 ( 2, 4 ),
$
which is the case if and only if
\begin{align}
 a ' \ge - m - 1.
\end{align}
\end{theorem}

\begin{corollary}\label{cr:strong_for_m=1}
Let
$
 \cE
$
be a sheaf bimodule on
$
 \X
$
such that
$
 \chi \lb \cU \rb = 2
$.
Then the exceptional collection \eqref{eq:EC} for
$
 m = 1
$
is strong if and only if
\begin{itemize}
\item
$
 W
$
is non-reduced and
$
 \deg D =\mbox {$0, 2$ or $4$} 
$,

\item
$
 W
$
is integral, or

\item
$
 W
$
is not irreducible and
$
 p \ge - 1
$.
\end{itemize}

Similarly, when
$
 \chi \lb \cU \rb = 1
$,
the exceptional collection \eqref{eq:EC} for
$
 m = 1
$
is strong if and only if
\begin{itemize}
\item
$
 W
$
is non-reduced and
$
 \deg D = \mbox{$1$ or $3$}
$,

\item
$
 W
$
is integral or

\item
$
 W
$
is not irreducible and
$
 p \ge - 1
$.
\end{itemize}
\end{corollary}
\begin{proof}
Immediately follows from \pref{cr:a'_b'}.
\end{proof}

\subsection{Quivers with relations and the moduli stack of relations}
 \label{sc:quiver}

If we set
\begin{align}
 (G_1, G_2, G_3, G_4) \coloneqq
  (v^* \cO_\X ( - m - 1 ), v^* \cO_\X( - m ), u^* \cO_\X ( - 1 ) \otimes \cU, \cU),
\end{align}
then
\begin{align}
\begin{aligned}
 \RHom_\X(F_2, F_3)
  &\simeq \RHom_\X(\cO_\X(- m), \phi^! \cO_\X ( - 1 )) \\
  &\simeq \RHom_\X(\cO_\X(- m), \cO_\X ( - 1 ) \otimes_{ \cO_\X } \cE) \\
  &\simeq \RHom_\X(\cO_\X(- m), v_* \lb u^* \cO_\X ( - 1 ) \otimes \cU \rb) \\
  &\simeq \RHom_W(v^* \cO_\X(- m), u^* \cO_\X ( - 1 ) \otimes \cU) \\
  &\simeq \RHom_W(G_2, G_3),
\end{aligned}
\end{align}
and similarly
\begin{align}
 \RHom_\X(F_i, F_j) \simeq \RHom_W(G_i, G_j)
\end{align}
for all
$
 1 \le i < j \le 4
$
except
$
 ( i, j ) = ( 3, 4 )
$.

Consider the case when
$
 \cU
$
is invertible,
$
 \deg \cU = 2
$,
the collection \eqref{eq:EC} for
$
 m = 1
$
is strong, and
$
 b - a = 0 = b ' - a '
$. Set
\begin{align}\label{eq:definition_of_L0_L1_L2}
 (L_0, L_1, L_2) \coloneqq
  (v^* \cO_\X(1), v ^{ * } \cO _{ \X } ( 1 ) \otimes u ^{ * } \cO _{ \X } ( - 1 ) \otimes \cU, u^* \cO_\X(1) ),
\end{align}
so that
\begin{align}
 \Hom(E_1, E_2) &\simeq H^0(L_0), &
 \Hom(E_2, E_3) &\simeq H^0(L_1), &
 \Hom(E_3, E_4) &\simeq H^0(L_2).
\end{align}
Consider the quiver
$
 Q ^{ 0 }
$
as in \pref{fg:quiver_of_nc_quadrics}.
\begin{figure}
\begin{align*}
\begin{psmatrix}[rowsep=35mm,colsep=35mm,mnode=circle]
 1 & 2 & 3 & 4
\psset{nodesep=3pt,arrows=->}
\ncline[offset=6pt]{1,1}{1,2}
 \lput*{N}(0.5){a_1}
\ncline[offset=-6pt]{1,1}{1,2}
 \lput*{N}(0.5){b_1}
\ncline[offset=6pt]{1,2}{1,3}
 \lput*{N}(0.5){a_2}
\ncline[offset=-6pt]{1,2}{1,3}
 \lput*{N}(0.5){b_2}
\ncline[offset=6pt]{1,3}{1,4}
 \lput*{N}(0.5){a_3}
\ncline[offset=-6pt]{1,3}{1,4}
 \lput*{N}(0.5){b_3}
\end{psmatrix}
\end{align*}
\caption{The quiver $Q^0$}
\label{fg:quiver_of_nc_quadrics}
\end{figure}
For
$
 i = 0, 1, 2
$,
fix an isomorphism of vector spaces
\begin{align}
 \vspan \lc a _{ i }, b _{ i } \rc \simeq H ^{ 0 } \lb L _{ i } \rb.
\end{align}
Under these isomorphisms, the following subspace is identified with a 2-dimensional linear subspace
$
 I \subset e _{ 4 } \bfk Q ^{ 0 } e _{ 1 }
$,
which automatically is a 2-sided ideal of
$
 \bfk Q ^{ 0 }
$.
\begin{align}\label{eq:F0_relations}
 \ker \lb H^0(L_2) \otimes H^0(L_1) \otimes H^0(L_0)
  \to H^0(L_0 \otimes L_1 \otimes L_2) \rb.
\end{align}
Then the endomorphism algebra of the full strong exceptional collection is isomorphic to the path algebra of
$
 Q ^{ 0 }
$
by the relations \eqref{eq:F0_relations}.

On the other hand, consider the 3-dimensional AS-regular cubic $\bZ$-algebra
$
 A = A (W, L_0, L_1, L_2)
$
associated to the admissible quadruple
\begin{align}\label{eq:admissible_quadruple}
 (W, L_0, L_1, L_2)
\end{align}
studied in \cite{MR2836401}. Then the relation \pref{eq:F0_relations}
is identified with the cubic relations
\begin{align}
 \ker \lb A _{ 2 3 } \otimes A _{ 1 2 } \otimes A _{ 0 1 }
 \to
 A _{ 0 3 } \rb,
\end{align}
which implies the derived equivalences
\begin{align}
 D ^{ b } \qgr \bS \lb \cE \rb
 \simeq
 D ^{ b } \module \bfk Q ^{ 0 } / I
 \simeq
 D ^{ b } \qgr A (W, L_0, L_1, L_2).
\end{align}
Combined with the arguments below, this shows the 2nd assertion of \pref{th:main}.
%

The moduli stack of relations of the quiver
$
 Q ^{ 0 }
$
is defined as
\begin{align} \label{eq:Mreltz}
 \Mreltz \coloneqq [V_1 \otimes V_2 \otimes V_3 \otimes V_4
  / \GL(V_1) \times \GL(V_2) \times \GL(V_3) \times \GL(V_4)],
\end{align}
where $V_i = \vspan \lc a_i, b_i \rc$
for $i=1,2,3$ and $V_4$ is a 2-dimensional vector space.
It is studied in detail in \cite{1403.0713}.
The generic stabilizer of $\Mreltz$ is the kernel $K_0 \cong (\bGm)^3$
of the map
\begin{align}
 Z \lb \prod_{i=1}^4 \GL(V_i) \rb
  \cong (\bGm)^4 \to \bGm, \quad
 (\lambda_i)_{i=1}^4
   \mapsto \lambda_1 \lambda_2 \lambda_3 \lambda_4.
\end{align}
We write the rigidified stack as
$
 \Mrelz \coloneqq \Mreltz / \mathrm{B}K_0.
$
The corresponding GIT quotient
\begin{align}
 \Proj \bfk[V_1 \otimes V_2 \otimes V_3 \otimes V_4]^{
  \SL(V_1) \times \SL(V_2) \times \SL(V_3) \times \SL(V_4)}
   \simeq \bP(2,4,4,6)
\end{align}
can also be interpreted as the SLOCC moduli space of 4 qubits
(cf.~\cite{1402.3768} and references therein).

Consider next the case when
$
 \cU
$
is invertible,
$
 \deg \cU = 1
$,
the collection \eqref{eq:EC} for
$
 m = 1
$
is strong, and
$
 b - a = 1 = b ' - a '
$.
Define the line bundles
$
 (L_0, L_1, L_2)
$
as in \eqref{eq:definition_of_L0_L1_L2},
so that
\begin{align}
 \Hom(E_1, E_2) &\simeq H^0(L_0), &
 \Hom(E_2, E_3) &\simeq H^0(L_1), &
 \Hom(E_3, E_4) &\simeq H^0(L_2).
\end{align}
Consider the quiver
$
 Q ^{ 1 }
$
defined as in \pref{fg:quiver_of_ncF1}.
\begin{figure}
\ \vspace{10mm}
\begin{align*}
\begin{psmatrix}[rowsep=35mm,colsep=35mm,mnode=circle]
 1 & 2 & 3 & 4
\psset{nodesep=3pt,arrows=->}
\ncline[offset=5pt]{1,1}{1,2}
 \lput*{N}(0.5){a_1}
\ncline[offset=-5pt]{1,1}{1,2}
 \lput*{N}(0.5){a_2}
\ncline[offset=0pt]{1,2}{1,3}
 \lput*{N}(0.5){a_7}
\ncline[offset=5pt]{1,3}{1,4}
 \lput*{N}(0.5){a_4}
\ncline[offset=-5pt]{1,3}{1,4}
 \lput*{N}(0.5){a_5}
\nccurve[angleA=50,angleB=130,offset=3pt]{1,1}{1,3}
 \lput*{N}(0.5){a_3}
\nccurve[angleA=-50,angleB=-130,offset=3pt]{1,2}{1,4}
 \lput*{N}(0.5){a_6}
\end{psmatrix}
\end{align*}
\ \vspace{10mm}
\caption{The quiver $Q ^{ 1 }$}
\label{fg:quiver_of_ncF1}
\end{figure}
Fix isomorphisms
\begin{align}
 \vspan\lc a _{ 1 }, a _{ 2 } \rc \simeq H ^{ 0 } \lb L _{ 0 } \rb,
 \vspan\lc a _{ 7 } \rc \simeq H ^{ 0 } \lb L _{ 1 } \rb,
 \vspan\lc a _{ 4 }, a _{ 5 } \rc \simeq H ^{ 0 } \lb L _{ 2 } \rb,
\end{align}
and lifts
\begin{align}
 \vspan\lc a _{ 3 } \rc \hookrightarrow \Hom \lb E _{ 0 }, E _{ 2 } \rb,
 \vspan\lc a _{ 6 } \rc \hookrightarrow \Hom \lb E _{ 1 }, E _{ 3 } \rb
\end{align}
of isomorphisms
\begin{align}
 \vspan\lc a _{ 3 } \rc
 \simto
 \coker \lb \Hom \lb E _{ 0 }, E _{ 1 } \rb \otimes \Hom \lb E _{ 1 }, E _{ 2 } \rb
 \to
 \Hom \lb E _{ 0 }, E _{ 2 } \rb \rb
\end{align}
and
\begin{align}
 \vspan\lc a _{ 6 } \rc
 \simto
 \coker \lb \Hom \lb E _{ 1 }, E _{ 2 } \rb \otimes \Hom \lb E _{ 2 }, E _{ 3 } \rb
 \to
 \Hom \lb E _{ 1 }, E _{ 3 } \rb \rb,
\end{align}
respectively.
From these choices one obtains a surjective homomorphism of
$
 \bfk ^{ 4 }
$-algebras from
$
 \bfk Q ^{ 1 }
$
to the endomorphism algebra of the full strong exceptional collection. The kernel of the homomorphism is a 3-dimensional linear subspace of
$
 e _{ 4 } \bfk Q ^{ 1 } e _{ 1 }
$, which automatically is a 2-sided ideal of the path algebra.


The moduli stack of relations of the quiver \eqref{fg:quiver_of_ncF1} is defined as follows.
Let
\begin{itemize}
\item
$
 V _{ 1 } = \vspan \lc a _{ 3 } \rc
$

\item
$
 V _{ 2 } = \vspan \lc a _{ 1 }, a _{ 2 } \rc
$

\item
$
 V _{ 3 } = \vspan \lc a _{ 7 } \rc
$

\item
$
 V _{ 4 } = \vspan \lc a _{ 4 }, a _{ 5 } \rc
$

\item
$
 V _{ 5 } = \vspan \lc a _{ 6 } \rc
$
\end{itemize}
and $V_6$ be a 3-dimensional vector space. Consider the group
\begin{align}
 G
 =
 \lb \Hom \lb V _{ 1 }, V _{ 2 } \otimes V _{ 3 } \rb \times \Hom \lb V _{ 5 }, V _{ 3 } \otimes V _{ 4 } \rb \rb
 \rtimes
 \prod _{ 1 \le i \le 6 } \GL ( V _{ i } ),
\end{align}
where the semi-direct product is defined by the obvious left action
\begin{align}
 \lb g _{ i } \rb _{ i = 1, \dots, 6 } \colon
 \lb \varphi, \psi \rb
 \mapsto
 \lb g _{ 2 } \otimes g _{ 3 } \circ \varphi \circ g _{ 1 } ^{ - 1 },
 g _{ 3 } \otimes g _{ 4 } \circ \psi \circ g _{ 5 } ^{ - 1 } \rb.
\end{align}
Then
\begin{align}\label{eq:definition_of_M_rel_1}
 \Mrelto
 \coloneqq [
 \lb V _{ 1 } \otimes V _{ 4 } \oplus V _{ 2 } \otimes V _{ 3 } \otimes V _{ 4 } \oplus V _{ 2 } \otimes V _{ 5 } \rb
 \otimes
 V _{ 6 }
  / G ].
\end{align}
The generic stabilizer of $\Mrelto$ is the direct product $K_1$
of the kernels of the maps
\begin{align}
 Z(\GL(V_1) \times \GL(V_4)) \to \bGm,
  \ (\lambda_1, \lambda_4) \mapsto \lambda_1 \lambda_4,
\end{align}
\begin{align}
 Z(\GL(V_2) \times \GL(V_5)) \to \bGm,
  \ (\lambda_2, \lambda_5) \mapsto \lambda_2 \lambda_5,
\end{align}
and
\begin{align}
 Z(\GL(V_2) \times \GL(V_3) \times \GL(V_4)) \to \bGm,
  \ (\lambda_2, \lambda_3, \lambda_5) \mapsto \lambda_2 \lambda_3 \lambda_5.
\end{align}
The rigidified stack will be denoted by
\begin{align}
 \Mrelo \coloneqq \Mrelto / \mathrm{B} K_1.
\end{align}


%
%
%
\section{Moduli stack of nonsingular admissible quadruples}
 \label{sc:ell}

\subsection{The moduli stack $\Mell$} \label{sc:Mell}

A \emph{nonsingular admissible quadruple}
$(E, L_0, L_1, L_2)$
consists of a smooth proper curve $E$ of genus 1
and three line bundles $(L_0, L_1, L_2)$ on $E$ such that
$
 L _{ i } \not\simeq L _{ j }
$
for
$
 i \neq j
$.
An isomorphism of nonsingular admissible quadruples
$(E, L_0, L_1, L_2)$ and $(E', L_0', L_1', L_2')$
consists of an isomorphism $\varphi \colon E \to E'$
of
$
 \bfk
$-schemes and isomorphisms
$
 \varphi _i \colon L_i \to \varphi^* L _{ i } '
$
of line bundles on $E$
for $i=0,1,2$.

Let $\Mellt$ be the category fibered in groupoids over the category of
$
 \bfk
$-schemes, where an object of its fiber category over a scheme
$
 S \to \Spec \bfk
$
is a collection $(\cC, \cL_0, \cL_1, \cL_2)$ of a family $\cC \to S$ of elliptic curves over $S$
and three line bundles $(\cL_0, \cL_1, \cL_2)$ on $\cC$, and
a morphism from $(\cC, \cL_0, \cL_1, \cL_2)$
to $(\cC', \cL_0', \cL_1', \cL_2')$
consists of an isomorphism
$\varphi \colon \cC \to \cC'$
of elliptic curves over $S$
and isomorphisms
$
 \varphi_i \colon \cL_i \to \varphi^* \cL_i'
$
of line bundles for $i=0,1,2$.
It is an algebraic stack,
which is a gerbe banded by $(\bGm)^3$,
and we write its rigidification as
$
 \Mell \coloneqq \Mellt / \mathrm{B} (\bGm)^3.
$

$\Mell$ is decomposed into connected components
by the degrees of line bundles.
The connected components of $\Mell$
parametrizing the quadruples $(E, L_0, L_1, L_2)$ with
\begin{align}
 \deg (L_0, L_1, L_2) = (2,2,2) \mbox{ and } (2,1,2)
\end{align}
will be denoted by
$\Mellz$ and $\Mello$
respectively.

\subsection{Sheaf bimodules and nonsingular admissible quadruples}

For $i=0,1$,
let
$
 \Mshi^0 \subset \Mshi
$
be the open substack of sheaf bimodules whose support $W$ is a smooth divisor of bidegree $(2,2)$.
There is a natural morphism
$
 \Phi_i \colon \Mshi^0 \to \Melli
$
which sends $\cE$ to $(W, L_0, L_1, L_2)$, where
$
 (L_0, L_1, L_2)
$
are the line bundles defined in \eqref{eq:definition_of_L0_L1_L2}. 
Conversely, given a nonsingular admissible quadruple
$
 (E, L_0, L_1, L_2),
$
since the line bundles
$L_2$ and $L_0$ are assumed to be non-isomorphic,
they together define a closed immersion
$
 \iota \colon E \hookrightarrow \X \times \X
$,
and the push-forward
$
 \cE \coloneqq \iota_* \lb L _{ 0 } ^{ - 1 } \otimes L_1 \otimes L _{ 2 } \rb
$
gives a sheaf bimodule on $\X$.
This gives the inverse morphism $\Melli \to \Mshi^0$,
so that $\Mshi^0$ and $\Melli$ are isomorphic to each other.

%
%
%
\subsection{From relations to nonsingular admissible quadruples}
 \label{sc:relell}

For $i=0,1$,
there is a natural morphism
$
 \Psi_i \colon \Melli \to \Mreli
$
sending a nonsingular admissible quadruple $(E, L_0, L_1, L_2)$
to the relation
coming from the full strong exceptional collection
\eqref{eq:EC}
for
$
 m = 1
$
of the derived category
$
 D ^{ b } \qgr \bS \lb \Phi_i^{-1}(E, L_0, L_1, L_2) \rb
$.
The inverse birational map from $\Mreli$ to $\Melli$ is given by
considering the moduli space of representations
of the quiver with relations.

For a 2-sided ideal
$
 I
$
of the quiver
$
 Q ^{ 0 }
$,
let
$
 \rep \lb Q ^{ 0 }, I \rb
$
be the stack of finite dimensional right
$
 \bfk Q ^{ 0 } / I
$-modules. Note that there are isomorphisms of abelian groups as follows.
\begin{align}\label{eq:K0_and_dimension_vector}
 K _{ 0 } \lb \rep \lb Q ^{ 0 }, I \rb \rb
 \simeq
 \bZ ^{ 4 };
 \quad
 [ M ]
 \mapsto
 \lb \dim _{ \bfk } M e _{ i } \rb _{ i = 1 } ^{ 4 }
\end{align}
For
$
 M \in \rep \lb Q ^{ 0 }, I \rb
$,
the corresponding element
$
 \lb \dim _{ \bfk } M e _{ i } \rb _{ i = 1 } ^{ 4 } \in \bZ ^{ 4 }
$
is called the \emph{dimension vector} of $M$.

Consider the map
$
 \theta \colon K _{ 0 } \lb \rep \lb Q ^{ 0 }, I \rb \rb \to \bZ
$
which is identified with
$
 \bZ ^{ 4 } \xto[]{ \lb - 3, 1, 1, 1 \rb } \bZ
$
under the isomorphism \eqref{eq:K0_and_dimension_vector}. A module
$
 M \in \rep \lb Q ^{ 0 }, I \rb
$
is said to be
$
 \theta
$-(semi) stable if the inequality
$
 \theta ( N ) < ( \le ) \theta ( M ) = 0
$
holds for any submodule
$
 0 \neq N \subsetneq M
$.

Let us consider the moduli stack
$
 \cN = \cN _{ Q ^{ 0 } } \lb \bsone, \theta \rb
$
of
$
 \theta
$-stable representations of the quiver $Q^0$ (without relations) of dimension vector $\bsone=(1, 1, 1, 1)$.
It follows from \cite{King} that
$
 \cN
$
admits a projective fine moduli scheme
which is described as the following GIT quotient, as we explain next.

For
$
 j = 0, 1, 2, 3
$, set
$
 U _{ j } \coloneqq \bfk
$
and
$
 H ' \coloneqq
 \prod _{ j } \GL ( U _{ j } )
 =
 \bG _{ m } ^{ 4 }
$.
Consider the affine space
\begin{align}
 \prod _{ a \in \lb Q ^{ 0 } \rb _{ 1 } } \Hom \lb U _{ t ( a ) }, U _{ s ( a ) } \rb
 =
 V _{ 0 } \times V _{ 1 } \times V _{ 2 },
\end{align}
where
\begin{align}
 V _{ i }
 \coloneqq
 \prod _{ \substack{ a \in \lb Q ^{ 0 } \rb _{ 1 }\\ s ( a ) = i } } \Hom \lb U _{ i + 1 }, U _{ i } \rb
 \simeq
 \bfk ^{ 2 }.
\end{align}
The group
$
 H '
$
naturally acts on it, but the small diagonal
$
 \bG _{ m } \subset H '
$
acts trivially. Hence we consider the induced action of the quotient group
$
 H \coloneqq H ' / \bG _{ m }
$.
In what follows, the symbol
$
 \chi ( - )
$
denotes the group of characters of a group $-$.
\begin{lemma}
There exists a canonical isomorphism
$
 \chi ( H ' ) \simto K _{ 0 } \lb Q ^{ 0 }, I \rb ^{ \vee }
$
which extends to an isomorphism of short exact sequences of abelian groups as follows, where the rightmost vertical map sends
$
 \id _{ \bG _{ m } }
$
to
$
 1
$.

\begin{align}
 \xymatrix{
 0 \ar[r]
 &
 \chi \lb H \rb \ar[r] \ar[d] ^{ \simeq }
 &
 \chi ( H ' ) \ar[r] \ar[d] ^{ \simeq }
 &
 \chi \lb \bG _{ m } \rb 
 =
 \bZ \id _{ \bG _{ m } } \ar[r] \ar[d] ^{ \simeq }
 &
 0\\
 0 \ar[r]
 & \Ker \varphi \ar[r]
 &
 K _{ 0 } \lb Q ^{ 0 }, I \rb ^{ \vee } \ar[r] ^{ \varphi \coloneqq \la -, \bsone \ra }
 &
 \bZ \ar[r]
 &
 0}
\end{align}
\end{lemma}

\begin{proof}
We only describe the canonical isomorphism. The rest of the proof is rather straightforward.

Take
$
 \theta \in \chi ( H ' )
$
and
$
 M \in \rep \lb Q ^{ 0 }, I \rb
$.
Consider the map
$
 t \colon \bG _{ m } \to \prod _{ i } \GL ( M e _{ i } )
$
which sends
$
 \lambda \in \bG _{ m }
$
to
$
 \lb M e _{ i } \xto[]{ \lambda \cdot } M e _{ i } \rb _{ i }
$.
Then we let the integer $n$ which is associated to the class
$
 [ M ] \in K _{ 0 } \lb Q ^{ 0 }, I \rb
$
to be the one defined by the equality
$
 \theta ( t ( \lambda ) ) = \lambda ^{ n }
$.
\end{proof}

Since
$
 \la \theta, \bsone \ra = 0
$,
we thus obtain a character of the group
$
 H
$,
which will also be denoted by
$
 \theta
$
by an abuse of notation.
It, in turn, defines a linearization of the space
$
 V _{ 1 } \times V _{ 2 } \times V _{ 3 }
$.

By the results of \cite{King}, the moduli stack
$
 \cN
$
is isomorphic to the GIT quotient
\begin{align}
 \lb V _{ 1 } \times V _{ 2 } \times V _{ 3 } \rb // _{ \theta }
 \lb \bG _{ m } ^{ 4 } / \bG _{ m } \rb
 \simeq
 \bP(V_1) \times \bP(V_2) \times \bP(V_3).
\end{align}

A relation $ [ I ] \in \Mrelz$ determines the moduli space
$
 \cN _{ I } = \cN _{ \lb Q ^{ 0 }, I \rb } \lb \bsone, \theta \rb \subset \cN
$
of representations of the quiver with relations
$
 ( Q ^{ 0 }, I )
$,
which is a complete intersection of two divisors
of multidegree $(1,1,1)$ in
$\bP(V_1) \times \bP(V_2) \times \bP(V_3)$.
It is an elliptic curve
if $I$ is sufficiently general.

Since
$
 \cN _{ I }
$
is a fine moduli scheme, it comes with the universal representation on it.
In particular, there are tautological line bundles
$
 M _{ 1 }, \dots, M _{ 4 }
$
corresponding to the 4 vertices of the quiver
$
 Q ^{ 0 }
$, which are unique up to simultaneous tensoring by a line bundle on
$
 \cN _{ I }
$.
Now
\begin{align}
 (E, L_0, L_1, L_2)
  \coloneqq (\cN _{ I }, M_2, M_2^{^\vee} \otimes M_3, M_3^{\vee} \otimes M_4)
\end{align}
gives a nonsingular admissible quadruple.

This induces a morphism
$
 \Phi _{ 0 }
$
to $\Mellz$ from the open substack of $\Mrelz$ consisting of points
$
 [ I ]
$
for which
$
 \cN _{ I }
$
is nonsingular and
$
 L _{ i } \not\simeq L _{ j }
$
for
$
 i \neq j
$,
which is a birational inverse to the morphism $\Psi_0$.
This is shown by checking
$
 \Psi _{ 0 } \Phi _{ 0 } = \id
$,
since both
$
 \Mrelz
$
and
$
 \Mellz
$
are irreducible smooth stacks of the same dimension 3 with trivial generic
stabilizers.
To see
$
 \Psi _{ 0 } \Phi _{ 0 }
 \lb E, L _{ 0 }, L _{ 1 }, L _{ 2 } \rb = \lb E, L _{ 0 }, L _{ 1 }, L _{ 2 } \rb
$,
note that for
$
 [ I ] = \Phi _{ 0 } \lb E, L _{ 0 }, L _{ 1 }, L _{ 2 } \rb
$
there exists the canonical embedding
$
 E \hookrightarrow \cN _{ I }
$.
In fact, if one chooses bases for
$
 H ^{ 0 } \lb E, L _{ i } \rb
$
for
$
 i = 1, 2, 3
$,
then one obtains it as the classifying morphism by regarding the collection
$
 \lb \cO _{ E }, L _{ 0 }, L _{ 0 } \otimes L _{ 1 }, L _{ 0 } \otimes L _{ 1 } \otimes L _{ 2 } \rb
$
and the set of bases as a family of stable representations of
$
 \lb Q ^{ 0 }, I \rb
$
parametrized by
$
 E
$.
From the description of
$
 \cN _{ I }
$
as a complete intersection, one concludes that the closed immersion
is actually an isomorphism.
The rest follows from the definition of the notion of classifying morphism.

The parallel arguments as above work for the moduli space
$
 \cN
$
of stable representations of the quiver $Q^1$ of dimension vector
$
 \bsone
$
(we use the same symbols as above for the counterparts).
For each choice of a generic stability condition, $\cN$ is the toric variety
obtained as the quotient of the stable locus of $\bA^7$
by a 3-dimensional torus $\bGm^3$.
The weight matrix of the torus action is given by
\begin{align}
\bordermatrix{
 & a_1 & a_2 & a_3 & a_4 & a_5 & a_6 & a_7 \cr
 t_2 & 1 & 1 & 1 & 0 & 0 & -1 & -1 \cr
 t_3 & 0 & 0 & 1 & -1 & -1 & 0 & 1 \cr
 t_4 & 0 & 0 & 0 & 1 & 1 & 1 & 0
},
\end{align}
whose kernel is given by the image of
\ \vspace{-10mm}\\
\begin{align}
\bordermatrix{
  & & & & \cr
 a_1 & -1 & -1 & 0 & 0 \cr
 a_2 & 1 & 0 & 0 & 0 \cr
 a_3 & 0 & 1 & 0 & 0 \cr
 a_4 & 0 & 0 & 1 & 0 \cr
 a_5 & 0 & 0 & -1 & -1 \cr
 a_6 & 0 & 0 & 0 & 1 \cr
 a_7 & 0 & -1 & 0 & -1
}.
\end{align}
One-dimensional cones of the corresponding fan
are generated by the rows of this matrix.
The choice of the stability condition determines
which cones of higher dimensions are included in the fan,
and it turns out that
for an appropriate choice of a generic stability condition,
the toric variety is isomorphic to $\bP^2 \times \bP^2$ blown-up at one point.
Although
the anti-canonical bundle is not ample,
it is big and globally generated,
and yields a flopping contraction.
For a general relations $I$ of $Q^1$,
the moduli space $\cN _{ I }$ is the intersection
of three divisors
of multidegree $(-1,0,0,1) \in \Hom(\bGm^{ ( Q ^{ 1 } ) _0}, \bGm)$.
Since a divisor of multidegree $(-1,0,0,1)$
is the cubic root of the anti-canonical divisor,
it is base point free and big.
Taking into account that it does not contract a divisor, by the Bertini's theorem we can conclude that the moduli space $\cN _{ I }$ is a connected elliptic curve for a general relation $[ I ]$.

The contraction
$
 \cN \to \bP^2 \times \bP^2
$
corresponds to the inversion of the arrow $a_7$,
which in turn corresponds to passage to the Beilinson quiver
for $\bP^2$.

\begin{remark}
When
$
 \cU
$
is not invertible, it is expected that the moduli space
$
 \cN _{ I }
$
for the ideal
$
 I
$
corresponding to $\cU$ is isomorphic to the fine moduli space of point modules, which is isomorphic to the scheme
$
 \bP _{ W } \cU
$
by \cite[Theorem 4.5.1]{MR2958936}.
At least there should be the classifying morphism
\begin{align}
 \bP \coloneqq \bP _{ W } \cU \to \cN _{ I }; \quad p \mapsto \Hom \lb \cT, \cO _{ p } \rb,
\end{align}
where the object $\cT$ should be the direct sum of four line bundles, whose successive differences in turn should give rise to admissible quadruples in the sense of \cite{MR2836401}.
\end{remark}

\section{Noncommutative derived (special) McKay correspondence}
\label{sc:McKay}

The \emph{McKay correspondence} is a name given to the relation
between various invariants of quotient stacks by finite groups
and those of crepant resolutions (if any) of their coarse moduli.
In characteristic 0,
a version of McKay correspondence as an equivalence
of derived categories of coherent sheaves has been established
in various cases,
starting with \cite{MR1752785}
and then by \cite{MR1824990} and many more.
This, in turn, is a particular case of the DK-hypothesis by Bondal, Orlov, Bridgeland, and Kawamata (see, e.g., \cite{MR3838122} and references therein).

The simplest case of the derived McKay correspondence appears in the surface
$
 A _{ 1 }
$-singularity.
It globalizes to the derived equivalence between the Hirzebruch surface
$
 \Sigma _{ 2 }
$
and the stack-theoretic weighted projective plane
$
 \bP \lb 1, 1, 2 \rb \coloneqq \ld \lb \bA^3 \setminus \bszero \rb / \bGm \rd
$,
both of which are crepant `resolutions'
of the weighted projective plane
$
 \bfP \lb 1, 1, 2 \rb
$:
\begin{align}
 \xymatrix{
 \Sigma _{ 2 } \ar[dr] _(.3){ \mbox{crepant resolution } } & & \bP \lb 1, 1, 2 \rb \ar[dl] ^(.35){ \mbox{ \quad\quad \'etale in codimension $1$}  }\\
 & \bfP \lb 1, 1, 2 \rb &}
\end{align}

On the one hand,
a sheaf bimodule $\cE$ on $\X$ gives an abelian category
$
 \qgr \bS (\cE),
$
which is a noncommutative deformation of
$
 \coh \Sigma _{ 2 }.
$
On the other hand,
a 3-dimensional AS-regular algebra $S$ generated by three elements of degree
$
 1, 1, 2
$,
classified by Stephenson \cite{MR1397387},
gives another abelian category
$
 \qgr S,
$
which is a noncommutative deformation of
$
 \qgr \bfk [ x, y, z ]
 \simeq
 \coh \bP \lb 1, 1, 2 \rb.
$
It is natural to ask if there is a derived equivalence
\begin{align}\label{eq:nc_derived_mckay}
 D ^{ b } \qgr \bS \lb \cE \rb \simeq D ^{ b } \qgr S.
\end{align}
In the rest of this section,
we will discuss
an example of a pair of $\cE$ and $S$
satisfying \pref{eq:nc_derived_mckay}.

\begin{remark}
Although there are essentially 3-dimensional moduli of deformations
for the abelian category
$
 \coh \Sigma _{ 2 }
$
as we have seen so far,
families of algebras on the list \cite[Theorem 2.11]{MR1927433}
apparently can not cover all of them.
This suggests one to study
3-dimensional AS-regular \emph{$\bZ$-algebras}
generated by three elements of degree
$
 1, 1, 2,
$
instead of AS-regular algebras
generated by three elements of degree
$
 1, 1, 2.
$
Note that a generic
3-dimensional AS-regular $\bZ$-algebra
generated by three elements of degree
$
 1, 1, 2
$
is isomorphic to a 3-dimensional
cubic AS-regular $\bZ$-algebra;
the relation in degree 2
turns the degree 2 generator
into a linear combination of quadratic monomials of degree 1 generators,
and one is left with cubic relations between degree 1 generators.
\end{remark}

For a 3-dimensional AS-regular algebra $S$
generated by three elements of degree
$
 1, 1, 2,
$
we let
$
 \pi \colon \grmod S \to \qgr S
$
be the quotient functor, and
write
$
 \cO(i) \coloneqq \pi \lb S ( i ) \rb
$
and
$
 \cO \coloneqq \cO(0).
$
The sequence
\begin{align}\label{eq:fsec_of_ncP112}
 \lb \cO, \cO ( 1 ), \cO ( 2 ), \cO ( 3 ) \rb
\end{align}
of objects of $\qgr S$
is a full strong exceptional collection in
$
 D ^{ b } \qgr S
$
(see, e.g., \cite[Corollary 18]{MR2641200}
and references therein).
We consider the algebra
\begin{align} \label{eq:q-Weyl}
 S = \bfk \la x, y, z \ra / \lb y x - x y, z x - \lambda x z, z y - y z \rb
\end{align}
appearing as the first item in \cite[Theorem 2.11]{MR1927433}.
The total morphism algebra of the collection
\eqref{eq:fsec_of_ncP112}
in this case
is isomorphic to the quotient of the path algebra
of the quiver in \pref{fg:quiver_of_nc_Sigma2}
by the relations
\begin{align}\label{eq:relations_stephenson_1}
 I = \lb 
 b _{ 2 } a _{ 1 } - a _{ 2 } b _{ 1 },
 b _{ 3 } a _{ 2 } - a _{ 3 } b _{ 2 },
 c _{ 2 } a _{ 1 } - \lambda a _{ 3 } c _{ 1 },
 c _{ 2 } b _{ 1 } - b _{ 3 } c _{ 1 } \rb.
\end{align}
\begin{figure}
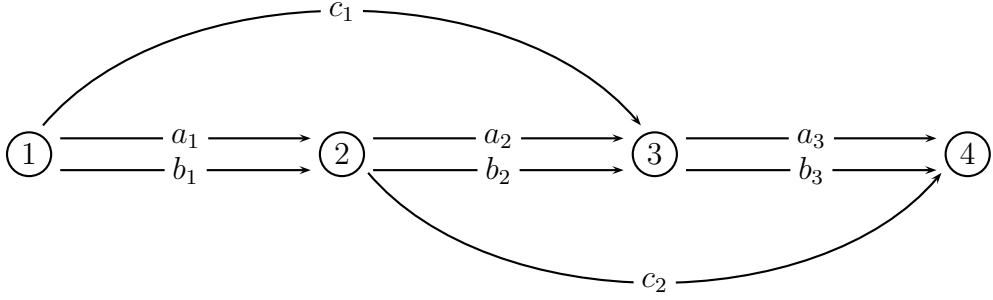

\ \vspace{10mm}
\begin{align*}
\begin{psmatrix}[rowsep=35mm,colsep=35mm,mnode=circle]
 1 & 2 & 3 & 4
\psset{nodesep=3pt,arrows=->}
\ncline[offset=6pt]{1,1}{1,2}
 \lput*{N}(0.5){a_1}
\ncline[offset=-6pt]{1,1}{1,2}
 \lput*{N}(0.5){b_1}
\ncline[offset=6pt]{1,2}{1,3}
 \lput*{N}(0.5){a_2}
\ncline[offset=-6pt]{1,2}{1,3}
 \lput*{N}(0.5){b_2}
\ncline[offset=6pt]{1,3}{1,4}
 \lput*{N}(0.5){a_3}
\ncline[offset=-6pt]{1,3}{1,4}
 \lput*{N}(0.5){b_3}
 \nccurve[angleA=50,angleB=130,offset=3pt]{1,1}{1,3}
 \lput*{N}(0.5){c_1}
\nccurve[angleA=-50,angleB=-130,offset=3pt]{1,2}{1,4}
 \lput*{N}(0.5){c_2}
\end{psmatrix}
\end{align*}
\ \\[12mm]
\caption{The quiver of type $\Sigma _{ 2 }$}
\label{fg:quiver_of_nc_Sigma2}
\end{figure}

Consider the sheaf bimodule
\begin{align}
 \cE
 =
 \cO _{ \Delta } ( 2 )
 \oplus
 \cO _{ \Gamma _{ \lambda } },
\end{align}
where
$
 \Gamma _{ \lambda }
$
is the graph of the automorphism
$
 \lambda \colon \X \to \X
$
given by
$
 \lb x : y \rb \mapsto \lb \lambda x : y \rb.
$
As explained in \pref{eq:E_and_F},
the quiver with relations
describing the total morphism algebra of the collection \eqref{eq:EC}
can be computed using
\begin{align}
 \lb F _{ 1 }, F _{ 2 }, F _{ 3 }, F _{ 4 } \rb
 =
 \lb \cO _{ \X } ( - 1 ), \cO _{ \X }, \cO _{ \X } ( 1 ) \oplus \cO _{ \X } ( - 1 ),
 \cO _{ \X } ( 2 ) \oplus \cO _{ \X } \rb.
\end{align}
For example, the arrows
$
 c _{ 1 }
$
and
$
 c _{ 2 }
$
correspond to the unique (up to scalar) morphisms
$
 \Hom _{ \X } \lb \cO _{ \X } ( - 1 ), \lambda _{ * } \cO _{ \X } ( - 1 )\rb
$
and
$
 \Hom _{ \X } \lb \cO _{ \X }, \lambda _{ * } \cO _{ \X } \rb
$,
which will be denoted by
$
 z _{ 1 }
$
and
$
 z _{ 2 }
$
respectively.

Choose a homogeneous coordinate
$
 x, y
$
of
$
 \X
$,
so that
$
 \Hom _{ \X } \lb  \cO _{ \X } ( - 1 ), \cO _{ \X } \rb = \vspan \lc x, y \rc
$.
A moment's reflection will convinces one that the dual gluing functor
$
 \phi ^{ ! }
$
sends
$
 x, y
$
to the morphisms
\begin{align}
 \begin{pmatrix}
 x & 0\\
 0 & \lambda ^{ - 1 } x
 \end{pmatrix},
 \begin{pmatrix}
 y & 0\\
 0 & y
 \end{pmatrix}
 \colon
 \cO _{ \X } ( 1 ) \oplus \cO _{ \X } ( - 1 )
 \to
 \cO _{ \X } ( 2 ) \oplus \cO _{ \X },
\end{align}
respectively.

Putting together, the endomorphism algebra of the full strong exceptional collection can be described as in
\pref{fg:endomorphisms_of_nc_Sigma2}.
\begin{figure}
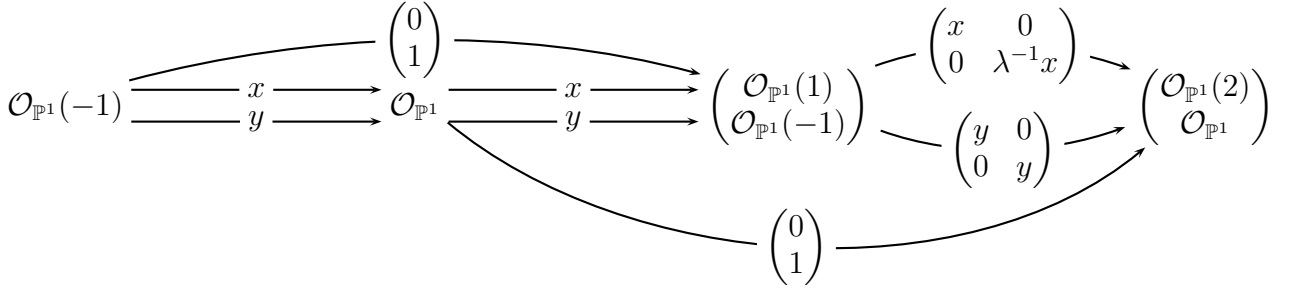

\ \vspace{10mm}
\begin{align*}
\begin{psmatrix}[rowsep=35mm,colsep=35mm,]
 \cO _{ \X } ( - 1 ) & \cO _{ \X } &
 \begin{pmatrix}
 \cO _{ \X } ( 1 )\\
 \cO _{ \X } ( - 1 )
 \end{pmatrix}
 &
 \begin{pmatrix}
 \cO _{ \X } ( 2 )\\
 \cO _{ \X }
 \end{pmatrix}
\psset{nodesep=3pt,arrows=->}
\ncline[offset=6pt]{1,1}{1,2}
 \lput*{N}(0.5){x}
\ncline[offset=-6pt]{1,1}{1,2}
 \lput*{N}(0.5){y}
\ncline[offset=6pt]{1,2}{1,3}
 \lput*{N}(0.5){x}
\ncline[offset=-6pt]{1,2}{1,3}
 \lput*{N}(0.5){y}
 \nccurve[angleA=20,angleB=160,offset=3pt]{1,3}{1,4}
 \lput*{N}(0.5){\begin{pmatrix}
 x & 0\\
 0 & \lambda ^{ - 1 } x
 \end{pmatrix}}
 \nccurve[angleA=-20,angleB=-160,offset=3pt]{1,3}{1,4}
 \lput*{N}(0.5){\begin{pmatrix}
 y & 0\\
 0 & y
 \end{pmatrix}}
 \nccurve[angleA=15,angleB=165,offset=3pt]{1,1}{1,3}
 \lput*{N}(0.5){
 \begin{pmatrix}
 0\\
 1
 \end{pmatrix}}
\nccurve[angleA=-40,angleB=-140,offset=3pt]{1,2}{1,4}
 \lput*{N}(0.5){
 \begin{pmatrix}
 0\\
 1
 \end{pmatrix}}
\end{psmatrix}
\end{align*}
\ \\[14mm]
\caption{Explicit description of the endomorphisms}
\label{fg:endomorphisms_of_nc_Sigma2}
\end{figure}

The relations \eqref{eq:relations_stephenson_1} can be readily seen from this. For example, the relation
$
 c _{ 2 } a _{ 1 } - \lambda a _{ 3 } c _{ 1 }
$
comes from the equality
\begin{align}
 \begin{pmatrix}
 0\\
 1
 \end{pmatrix}
 \circ x
 =
 \lambda
 \begin{pmatrix}
 x & 0\\
 0 & \lambda ^{ - 1 } x
 \end{pmatrix}
 \circ
 \begin{pmatrix}
 0\\
 1
 \end{pmatrix}
\end{align}
among the paths from
$
 \cO _{ \X } ( - 1 )
$
to
$
 \cO _{ \X } ( 2 ) \oplus \cO _{ \X }
$.

More generally, for
$
 d \ge 3
$,
there exists a fully faithful functor
\begin{align}\label{eq:embedding_d}
 D ^{ b } \coh \Sigma _{ d } \hookrightarrow D ^{ b } \coh \bP ( 1, 1, d )
\end{align}
whose essential image is given by
\begin{align}
 \la \cO, \cO(1), \cO(d), \cO(d+1) \ra
  \subset D^b \coh \bP(1,1,d)
  = \la \cO, \cO(1), \cO(2), \ldots, \cO(d+1) \ra.
\end{align}
This is a global version of the derived special McKay correspondence
for the cyclic quotient singularity
$
 \frac{1}{d} \lb 1, 1 \rb
$.
Similarly,
the category $D^b \qgr S$
for any 3-dimensional AS-regular algebra $S$ generated by elements of degrees
$
 1, 1, d
$
has an admissible subcategory
$
 \la \cO, \cO(1), \cO(d), \cO(d+1) \ra
$
admitting a semiorthogonal decomposition
into two copies of $D^b \coh \X$.
It is an interesting problem
to see if the integral kernel
for the dual gluing functor
is given by a locally free sheaf bimodule $\cE$,
so that one obtains a fully faithful functor
\begin{align}
 D^b \qgr \bS(\cE) \hookrightarrow D^b \qgr S,
\end{align}
which is a non-commutative generalization of \pref{eq:embedding_d}.
A calculation parallel to the one given above
shows that this is the case for \pref{eq:q-Weyl}.


\bibliographystyle{amsalpha}
\bibliography{bibs}

\end{document}